\documentclass[a4paper,reqno]{amsart}
\usepackage{amscd,amsfonts,amssymb}
\usepackage{tikz}
\usepackage{pgfplots}
\usetikzlibrary{automata,positioning,calc,trees}
\usetikzlibrary{intersections,pgfplots.fillbetween}
\usepackage{graphicx,tikz}
\usepackage{pgf,tikz,pgfplots}
\usepackage{mathrsfs}
\usetikzlibrary{arrows}
\usepackage{enumerate}
\usepackage[shortlabels]{enumitem}
\usepackage{mathrsfs}
\usepackage{amssymb,amsmath,amsthm,color}
\usepackage{relsize}
\usepackage{caption,subcaption}
\usepackage{hyperref}
\usepackage{cleveref}
\usepackage[alphabetic,bibtex-style]{amsrefs}
\usepackage{url}
\usepackage{setspace}
\usepackage{float}
\usepackage{makecell}

\newtheorem{thm}{Theorem}[section]
\newtheorem{lemma}{Lemma}[section]
\newtheorem{definition}{Definition}[section]

\newtheorem{theorem}[thm]{Theorem}
\newtheorem{remark}[thm]{Remark}
\newcounter{old}

\newtheorem{corollary}[thm]{Corollary}

\numberwithin{equation}{section}
\newtheorem{atheorem}{Theorem}

\newcommand{\R}{{\mathbb {R}}}

\newcommand{\N}{{\mathbb N}}
\newcommand{\Z}{{\mathbb Z}}
\newcommand{\C}{{\mathbb C}}

\newcommand{\supp}{\operatorname{supp}}

\BibSpec{article}{%
	+{}  {\PrintAuthors}                {author}
	+{,} { \textit}                     {title}
	+{.} { }                            {part}
	+{:} { \textit}                     {subtitle}
	+{,} { \PrintContributions}         {contribution}
	+{.} { \PrintPartials}              {partial}
	+{,} { }                            {journal}
	+{}  { \textbf}                     {volume}
	+{}  { \PrintDatePV}                {date}
	+{,} { \issuetext}                  {number}
	+{,} { \eprintpages}                {pages}
	+{,} { }                            {status}
	+{,} { \url}                        {url}    
	+{,} { \PrintDOI}                   {doi}
	+{,} { available at \eprint}        {eprint}
	+{}  { \parenthesize}               {language}
	+{}  { \PrintTranslation}           {translation}
	+{;} { \PrintReprint}               {reprint}
	+{.} { }                            {note}
	+{.} {}                             {transition}
}
\allowdisplaybreaks

\begin{document}
	\title{Bilinear spherical maximal function with fractal dilations}
	
	\author[S.S. Choudhary]{Surjeet Singh Choudhary}
	\address{Surjeet Singh Choudhary\\
		National Center for Theoretical Sciences\\
		National Taiwan University\\
		Taipei-106, Taiwan.}
	\email{surjeet19@ncts.ntu.edu.tw}
	
	\author[C.Y. Shen]{Chun-Yen Shen}
	\address{Chun-Yen Shen\\
		Department of Mathematics\\
		National Taiwan University\\
		Taipei-106, Taiwan.}
	\email{cyshen@ntu.edu.tw}
	
	\author[S. Shrivastava]{Saurabh Shrivastava}
	\address{Saurabh Shrivastava\\
		Department of Mathematics\\
		Indian Institute of Science Education and Research Bhopal\\
		Bhopal-462066, India.}
	\email{saurabhk@iiserb.ac.in}

	\begin{abstract}
		In this paper, we investigate the $L^p-$boundedness of the bilinear spherical maximal function associated with a general set of dilations $E\subset\R_+$. We quantify the range of $L^p-$boundedness in terms of a dilation-invariant notion of the upper Minkowski dimension of the set $E$. A particular case of this study settles an open question of $L^p-$boundedness of the maximal lacunary bilinear spherical function in borderline cases $p_1=1$ or $p_2=1$ in dimension $d\geq4$.
	\end{abstract}
	
	\maketitle
	\tableofcontents
	\section{Introduction}
	The spherical average of $f:\R^d\rightarrow \C,~d\geq 2,$ is defined by 
	\[A_tf(x)=\int_{\mathbb S^{d-1}}f(x-ty)d\sigma(y), ~t>0,\]
	where $d\sigma$ denotes the normalized surface measure on the sphere $\mathbb S^{d-1}$.
	Given a set $ E\subset \R_+$, the corresponding spherical maximal operator is defined by
	\[M_Ef(x)=\sup_{t\in E}|A_tf(x)|.\]
	Stein \cite{SphStein} used Fourier analytic methods to prove that the maximal operator $M_{\R_+}f$ is bounded in $L^p$ for $p>\frac{d}{d-1}$ for $d\geq 3$. This range of $p$ is sharp up to the end point. The maximal operator $M_{\R_+}f$  is referred to as the full spherical maximal operator, and it is denoted by $M_{full}f$. The Fourier analytical approach has a limitation in resolving the problem in dimension $d=2$. Bourgain \cite{CirBourgain} introduced a geometric method to settle the case of dimension $d=2$. He showed that $M_{full}f$ is bounded in $L^p$ for $p>2$ for $d=2$. At the endpoint $p=\frac{d}{d-1}$, Bourgain \cite{sphBourgain} proved restricted weak-type bounds in dimensions $d\geq3$, while Seeger, Tao, and Wright~\cite{STW} showed that the restricted weak-type inequality at the endpoint $p=\frac{d}{d-1}$ fails in dimension $d=2$. 
	
	Under suitable restrictions on the set $E\subset\R_+$, the $L^p-$boundedness of $M_E$ extends to a larger range of $L^p-$spaces. The range of $p$ depends on the geometry and appropriate notion of size of the set $E$. For a general set $E\subset\R_+$, Seeger, Wainger and Wright \cite{SWW1} obtained sharp $L^p-$bounds for $p>p(E)$, where $p(E)$ depends on the dilation-invariant notion of the Minkowski dimension. The endpoint cases were resolved in \cite{STW}. A particularly interesting case arises when we consider $E$ to be a lacunary set. In this case, the $L^p-$bounds of the operator $M_E$ (denoted by $M_{lac}$), hold for all $p>1$, see \cites{lacsphCalderon, CoifmanWeiss, RubioDuoandi}. It is conjectured that the lacunary maximal operator $M_{lac}$ is of weak-type $(1,1)$; however, the problem remains open to date. We refer to \cite{CladekKrause} for recent developments in this direction. Further, we would like to refer the reader to Lacey \cite{sparseLacey} for the sharp range of sparse domination results for the maximal operators $M_{full}$ and $M_{lac}$. The continuity estimates and $L^p-$improving bounds for local variants of spherical maximal function play key roles in proving sparse domination results. 
	
	When $E\subset[1,2]$, the corresponding maximal operator $M_E$ is referred to as the local maximal operator. Due to the local nature of the operator, it satisfies improving $L^p\to L^q-$bounds for a range of $p$ and $q$ with $q\geq p$. Such estimates for the local operator $M_E$, in terms of the upper Minkowski dimension and the Assouad spectrum of $E$, have been studied in \cites{AHRS, RoosSeeger, LRZZ}. In these results, the Fourier decay of surface measures plays a key role, and the fractal geometry of $E$ is captured through the covering numbers and Minkowski dimension. 
	
	We now recall the dilation-invariant notion of the Minkowski dimension defined in \cite[(1.3)]{SWW1}.
	\begin{definition}\label{def}
		Let $E\subset\R_+$ and $E_k=E\cap[2^{-k},2^{-k+1}]$, $k\in\Z$. We say $E$ has upper Minkowski dimension $\beta$ if
		\[\limsup_{\delta\to0}\sup_{k\in\Z}\frac{\log(1+ N(E_k,2^{-k}\delta))}{\log(1+\delta^{-1})}=\beta,\]
		where $N(T,\delta)$ denotes the minimum number of $\delta-$length intervals needed to cover $T$.
	\end{definition}
	Note that the definition of upper Minkowski dimension is uniform in dyadic scales. In particular, if $E$ has dimension $\beta$ in the sense above, then
	\[N(E_k,2^{-k}\delta)\leq \delta^{-\beta},\qquad 0<\delta<1,\,\forall\, k\in\Z.\]
	\subsection{Bilinear spherical maximal function}
	The bilinear spherical maximal function associated with a given set $E\subset \R_+$ is defined by
	\[\mathcal{M}_E(f_1,f_2)(x):=\sup_{t\in E}|\mathcal{A}_t(f_1,f_2)(x)|,\]
	where the bilinear spherical average $\mathcal{A}_t(f_1,f_2)(x)$ is given by 
	\[\mathcal{A}_t(f_1,f_2)(x)=\int_{\mathbb S^{2d-1}}f_1(x-ty_1)f_2(x-ty_2)\;d\sigma(y_1,y_2).\]
	The bilinear spherical averages first appeared in \cite{GGIPS}.
	The $L^p-$estimates for the full maximal operator $\mathcal{M}_{full}:=\mathcal{M}_{\R_+}$ were studied in \cites{BGHHO, GHH, HHY}. The sharp range of $L^p-$boundedness, including the endpoint cases, was obtained by Jeong and Lee~\cite{JeongLee} for dimensions $d\geq2$. They introduced the method of slicing and gave a simple proof of optimal $L^p-$estimates for the operator $\mathcal{M}_{full}$. The slicing argument provides a decomposition of  $\mathbb{S}^{2d-1}$ into a family of lower dimensional spheres. More precisely, Jeong and Lee \cite{JeongLee} showed that if $F$ is a continuous function defined
	on $\R^{2d}, d\geq 2$ and $(y,z)\in\R^d\times\R^d$, then
	\begin{equation}\label{slicing}
		\int_{\mathbb{S}^{2d-1}}F(y,z)\;d\sigma(y,z)= C_d \int_{B^d(0,1)}\int_{\mathbb{S}^{d-1}}F(y,\sqrt{1-|y|^2}z)\;d\sigma_{d-1}(z)(1-|y|^2)^{\frac{d-2}{2}}dy.
	\end{equation}
	The formula as above yields the following pointwise estimate
	\begin{align}\label{slicing1}\mathcal{M}_{full}(f,g)(x)\lesssim M_{HL}f(x)M_{full}g(x),	\end{align}
	where $M_{HL}f$ is the Hardy-Littlewood maximal function. Using the symmetry of the bilinear spherical maximal function in $f$ and $g$, the boundedness region for the operator $\mathcal M_{full}$ is determined by using $L^p-$estimates of Hardy-Littlewood maximal function $M_{HL}f(x)$ and spherical maximal function $M_{full}g(x)$ for $d\geq 2$. In dimension $d=1$, the optimal range of boundedness of $\mathcal M_{full}$ was obtained in \cites{ChristZhou, DosidisRamos, BCSS}.
	
	Jeong and Lee~\cite{JeongLee} also obtained $L^p-$improving estimates for the local version of bilinear spherical maximal function using the slicing method. Later, in~\cite{BCSS}, the authors improved the range of $L^p-$improving estimates for the local bilinear spherical maximal function. Observe that for a general set $E$, the slicing argument always leads to pointwise domination of the bilinear maximal function $\mathcal{M}_{E}(f,g)$ by a product of the Hardy-Littlewood maximal function and the full spherical maximal function as in the estimate~\eqref{slicing1}. Consequently, it does not exploit the geometry of $E$ and therefore does not enlarge the $L^p$-boundedness range beyond that of the full operator.
	
	The $L^p-$estimates for the bilinear lacunary spherical maximal operator $\mathcal{M}_{lac}:=\mathcal{M}_E$, where $E$ is a lacunary set, were obtained by Borges and Foster \cite{BorgesFoster} for dimensions $d\geq 2$, and by Christ and Zhou~ \cite{ChristZhou} for dimension $d=1$. The proof for dimensions $d\geq 2$ relies mainly on an $L^2\times L^2\rightarrow L^1-$decay estimate in the frequency parameter for single-scale average $\mathcal{A}_1(f,g)$, where either $f$ or $g$ is frequency-localized in an appropriate annular region formed by Littlewood--Paley decomposition. Then, the desired estimates for the bilinear lacunary maximal function are obtained by a bootstrapping type argument. This bootstrapping argument, however, breaks down at the endpoint when either $p_1=1$ or $p_2=1$. Consequently, in these endpoint cases, the boundedness range obtained for $\mathcal{M}_{lac}$ does not improve upon that for $\mathcal{M}_{full}$. The decay estimate at the point $(2,2,1)$ is obtained by a combination of the slicing argument \eqref{slicing} and Fourier decay property of $\widehat{d\sigma}$.
	
	Due to unavailability of the slicing method and lack of appropriate Fourier decay of $\widehat{d\sigma}$ in dimension $d=1$, it is significantly more difficult to obtain $L^p-$estimates for $\mathcal{M}_{lac}$. Christ and Zhou~ \cite{ChristZhou} overcame these issues and proved trilinear smoothing estimates for bilinear averaging operators along certain curves, including the circle. The trilinear smoothing estimates play a key role in proving decay estimates at the point $(2,2,1)$ for frequency-localized pieces and finally in obtaining the desired $L^p-$estimates for $\mathcal{M}_{lac}$. Here we summarize the results for the bilinear lacunary maximal operator  $\mathcal{M}_{lac}$.
	\begin{atheorem}\label{lac1}\cites{ChristZhou,BorgesFoster}
		Let $1\leq p_1,p_2\leq\infty$ with $\frac{1}{p}=\frac{1}{p_1}+\frac{1}{p_2}$. Then, $\mathcal{M}_{lac}$ maps $L^{p_1}(\mathbb{R}^{d})\times L^{p_2}(\mathbb{R}^{d})$ to $L^{p}(\mathbb{R}^{d})$ if
		\begin{enumerate}[(i)]
			\item $d=1$ and $1< p_1,p_2\leq\infty$.
			\item $d\geq2$ and $1< p_1,p_2\leq\infty$, or $p_1=1$ and $p_2\in(\frac{d}{d-1},\infty)$, or $p_2=1$ and $p_1\in(\frac{d}{d-1},\infty)$.
		\end{enumerate}
		Moreover, the weak-type boundedness $\mathcal{M}_{lac}:L^{1}(\mathbb{R}^{d})\times L^{\infty}(\mathbb{R}^{d})\rightarrow L^{1,\infty}(\mathbb{R}^{d})$ and $\mathcal{M}_{lac}:L^{\infty}(\mathbb{R}^{d})\times L^{1}(\mathbb{R}^{d})\rightarrow  L^{1,\infty}(\mathbb{R}^{d})$ hold. 
	\end{atheorem}
	\begin{figure}[H]
		\centering
		\begin{tikzpicture}[scale=3]
			\fill[lightgray] (0,0)--(1,0)--(1,1)--(0,1)--cycle;
			\draw[thin][->]  (0,0)node[left]{$O$} --(1.15,0) node[right]{$\frac{1}{p_1}$};
			\draw[thin][->]  (0,0) --(0,1.2) node[left]{$\frac{1}{p_2}$};
			\draw [densely dotted] (1,0) node[below]{$A$} --(1,1) node[right]{$(1,1)$}--(0,1)node[left]{$B$};       
			\node at (1,0) {\tiny{$\circ$}};
			\node at (0,1) {\tiny{$\circ$}};
			\node at (0.5,-0.3) {\textit{ $d=1$.}};
		\end{tikzpicture}
		\qquad
		\begin{tikzpicture}[scale=3]
			\fill[lightgray] (0,0)--(1,0)--(1,1)--(0,1)--cycle;
			\draw[thin][->]  (0,0)node[left]{$O$} --(1.15,0) node[right]{$\frac{1}{p_1}$};
			\draw[thin][->]  (0,0) --(0,1.2) node[left]{$\frac{1}{p_2}$};
			\draw [thin] (1,0) node[below]{$A$} --(1,2/3)node[right]{$Q_0$};
			\draw[thin] (0,1)node[left]{$B$}--(2/3,1)node[above]{$R_0$};
			\draw[densely dotted] (1,2/3)--(1,1)--(2/3,1);
			\node at (1,0) {\tiny{$\circ$}};
			\node at (0,1) {\tiny{$\circ$}};
			\node at (1,2/3) {\tiny{$\circ$}};
			\node at (2/3,1) {\tiny{$\circ$}};
			\node at (0.5,-0.3) {\textit{$d\geq2$.}};
		\end{tikzpicture}
		\caption{Boundedness regions for the operator $\mathcal M_{lac}$ in \Cref{lac1}.}
	\end{figure}
	In dimension $d=1$, Christ and Zhou~\cite{ChristZhou} showed that the strong-type $L^p-$bounds for $\mathcal{M}_{lac}$ do not hold if either $p_1=1$ or $p_2=1$. However, in higher dimensions, $L^p-$boundedness of $\mathcal{M}_{lac}$ is known to hold for the range $p_1=1$ and $p_2\in(\frac{d}{d-1},\infty)$, or $p_2=1$ and $p_1\in(\frac{d}{d-1},\infty)$. The question concerning the remaining points on the line segments $p_1=1$ and $p_2=1$ remained unresolved. In this paper, we address this problem and extend the range of $L^p-$boundedness of $\mathcal{M}_{lac}$ to the borderline cases in  $d\geq4$ except for the endpoint $(1,1)$. Indeed, we obtain new results for the bilinear spherical maximal function associated with general sets.
	
	The study of local bilinear spherical maximal function with a general set $E\subset[1,2]$ was carried out by Borges, Foster and Ou \cite{BFO}. They proved Sobolev smoothing estimates at the point $(2,2,2)$ for the local maximal operator along hypersurfaces that exhibit Fourier decay of order $(1+|\xi|)^{-\alpha}$ with $\alpha>\frac{d}{2}$. They also obtained necessary conditions for the boundedness of $\mathcal{M}_E$ in terms of the upper Minkowski dimension of the underlying fractal set. In addition to these dimension-dependent conditions, we have the elementary necessary conditions $p_1,p_2\geq1$.
	
	The Sobolev smoothing estimates give a simplified derivation of the continuity estimates for such local maximal operators. They also deduced $L^p-$improving bounds and sparse bounds for the associated bilinear maximal operators. However, when restricted to spherical maximal functions, these results do not enlarge the $L^p$-boundedness region beyond that of $\mathcal M_{full}$.
	
	The sparse bounds for the bilinear spherical maximal function $\mathcal{M}_{full}$ were studied by Borges et al. in \cite{BFOPZ}. Extensions of these results to some general surfaces have also been explored; for instance, see \cites{CGHHS, LS} for $L^p-$estimates and \cites{RSS, PS} for sparse bounds. Additionally, the multilinear extension of the maximal function associated with spheres and other general surfaces has been studied in \cites{Dosidis, SS, GHHP, CLS, Gao}.
	\section{Main results}\label{sec:results}
	Our approach replaces the slicing argument with a frequency-side decomposition of the associated bilinear multiplier that exploits the geometry of the dilation set. We intend to quantify the range of exponents $(p_1, p_2, p)$ for the $L^{p_1}\times L^{p_2}\rightarrow L^p-$boundedness of the operator $\mathcal M_E$ in terms of the upper Minkowski dimension of the set $E$. We exploit the Fourier decay property of the surface measure in a refined manner and perform a scale-adapted discretization of the set $E$ to capture fractal geometry of the underlying set. These ideas, along with suitable adaptation of bilinear Calder\'{o}n-Zygmund theory, yield new results for the maximal operator $\mathcal M_E$. In order to describe our results, we need the following notation.
	
	Let $\Omega({P_1,P_2,\dots,P_k})$ denote the convex hull of points $P_1,P_2,\dots,P_k\in\R^2$.  Let $d\geq2$ and set $O = (0, 0),\; A = (1, 0),\; B = (0, 1),\; P_0=\big(\frac{2d-1}{2d},\frac{2d-1}{2d}\big),\; Q_0=\big(1,\frac{d-1}{d}\big), \;\text{and}~ R_0=\big(\frac{d-1}{d},1\big)$. For $0\leq\beta\leq1$ consider the points
	\[P=P(d,\beta)=\begin{cases} 
		\left(\frac{2d-2}{2d-2+\beta},\frac{2d-2}{2d-2+\beta}\right) & \beta<\frac{2d-2}{2d-1}, \\
		P_0 & \beta\geq \frac{2d-2}{2d-1},
	\end{cases}\]
	\[Q=Q(d,\beta)=\begin{cases} 
		\left(1,\frac{d-3}{d-3+\beta}\right) & \beta<\frac{d-3}{d-1}, \\
		Q_0 & \beta\geq \frac{d-3}{d-1} \text{ or }d=2,3,
	\end{cases}\]
	and 
	\[R=R(d,\beta)=\begin{cases} 
		\left(\frac{d-3}{d-3+\beta},1\right) & \beta<\frac{d-3}{d-1}, \\
		R_0 & \beta\geq \frac{d-3}{d-1} \text{ or }d=2,3.\end{cases}\]
	The following are the main results of the paper.
	\begin{theorem}\label{main}
		Let $d\geq2$ and $0<p_1,p_2,p\leq\infty$ with $\frac{1}{p_1}+\frac{1}{p_2}=\frac{1}{p}$.  Assume that $E\subset\R_+$ has an upper Minkowski dimension $\beta$. Then the operator $\mathcal{M}_E$ is bounded from $L^{p_1}(\mathbb{R}^{d})\times L^{p_2}(\mathbb{R}^{d})$ into $L^{p}(\mathbb{R}^{d}),$ where 
		\begin{enumerate}
			\item $d=2,3$, and $(\frac{1}{p_1},\frac{1}{p_2})\in \Omega(O,A,Q_0,P,R_0,B)\setminus \{A,B,PQ_0,PR_0\}$.
			\item $d\geq4$, and $(\frac{1}{p_1},\frac{1}{p_2})\in \Omega(O,A,Q,P,R,B)\setminus \{A,B,PQ,PR\}$.
		\end{enumerate}
	\end{theorem}
	\begin{theorem}\label{main2}
		Let $d=1$ and $0<p_1,p_2,p\leq\infty$ with $\frac{1}{p_1}+\frac{1}{p_2}=\frac{1}{p}$. Assume that $E\subset\R_+$ has an upper Minkowski dimension $\beta$ and define
		\begin{align*}
			D_1=\begin{cases} 
				\left(1,0\right) & \beta<\frac{1}{2}, \\
				\left(\frac{1}{2\beta},0\right) & \beta\geq \frac{1}{2}, 
			\end{cases},\qquad D_2= \left(\frac{1}{2},\frac{1}{2}\right)\quad\text{and}\quad D_3= \begin{cases} 
				\left(0,1\right) & \beta<\frac{1}{2}, \\
				\left(0,\frac{1}{2\beta}\right) & \beta\geq \frac{1}{2}.
			\end{cases}
		\end{align*}        
		Then, the operator $\mathcal{M}_E$ is bounded from $L^{p_1}(\mathbb{R})\times L^{p_2}(\mathbb{R})$ into $L^{p}(\mathbb{R})$ when $(\frac{1}{p_1},\frac{1}{p_2})\in \Omega(O,D_1,D_2,D_3)\setminus \{D_1D_2, D_2D_3\}$.
	\end{theorem}
	\begin{figure}[H]
		\centering
		\begin{subfigure}[t]{0.48\textwidth}
			\centering
			\begin{tikzpicture}[scale=3.5]
				\fill[lightgray] (0,0)--(1,0)--(1,2/3)--(2/3,1)--(0,1)--cycle;
				\fill[gray] (1,2/3)--(1,4/5)--(24/25,24/25)--(4/5,1)--(2/3,1)--(1,2/3);
				
				\draw[thin][->] (0,0) node[left]{$O$}
				--(1.15,0) node[right]{$\frac{1}{p_1}$};
				\draw[thin][->] (0,0)
				--(0,1.2) node[left]{$\frac{1}{p_2}$};
				
				\draw[thin] (1,0) node[below]{$A$}
				--(1,2/3) node[right]{$Q_0$}
				--(1,4/5) node[right]{$Q$};
				
				\draw[thin] (0,1) node[left]{$B$}
				--(2/3,1) node[above]{$R_0$}
				--(4/5,1) node[above]{$R$};
				
				\draw[densely dotted] (1,4/5)
				--(24/25,24/25) node[right]{$P$}
				--(4/5,1);
				
				\draw[densely dotted] (1,0)--(1,1)--(0,1);
				\draw[densely dotted] (1,2/3)--(2/3,1);
				
				\node at (1,0) {\tiny{$\circ$}};
				\node at (0,1) {\tiny{$\circ$}};
				\node at (1,4/5) {\tiny{$\circ$}};
				\node at (4/5,1) {\tiny{$\circ$}};
			\end{tikzpicture}
			\caption{$d\geq 2$.}
			\label{fig:dgeq2}
		\end{subfigure}
		\hfill
		\begin{subfigure}[t]{0.48\textwidth}
			\centering
			\begin{tikzpicture}[scale=3.5]
				\fill[lightgray] (0,0)--(0.7,0)--(0.5,0.5)--(0,0.7)--cycle;
				
				\draw[thin][->] (0,0) node[left]{$O$}
				--(1.15,0) node[right]{$\frac{1}{p_1}$};
				\draw[thin][->] (0,0)
				--(0,1.2) node[left]{$\frac{1}{p_2}$};
				
				\draw[densely dotted] (1,0)
				--(1,1) node[right]{$(1,1)$}
				--(0,1) node[left]{$B$};
				
				\node at (0.7,0) {\tiny{$\circ$}};
				\node at (0,0.7) {\tiny{$\circ$}};
				\node at (0.5,0.5) {\tiny{$\circ$}};
				
				\draw[densely dotted] (0.7,0) node[below]{$D_1$}
				--(0.5,0.5) node[right]{$D_2$}
				--(0,0.7) node[left]{$D_3$};
			\end{tikzpicture}
			\caption{$d=1$.}
			\label{fig:d1}
		\end{subfigure}
		
		\caption{Boundedness regions for the operator $\mathcal M_E$.}
		\label{fig:boundedness-regions}
	\end{figure}
	Note that the bilinear lacunary spherical maximal function corresponds to the case $\beta=0$. As a consequence of \Cref{main}, we obtain $L^p-$boundedness of the lacunary bilinear spherical maximal operator $\mathcal M_{lac}$ in the borderline cases in dimension $d\geq4$.
	\begin{theorem}\label{lac2}
		Let $d\geq4$ and $1\leq p_1,p_2\leq\infty$ with $\frac{1}{p}=\frac{1}{p_1}+\frac{1}{p_2}$. Then, $\mathcal{M}_{lac}$ maps $L^{p_1}(\mathbb{R}^{d})\times L^{p_2}(\mathbb{R}^{d})$ to $L^{p}(\mathbb{R}^{d})$ except for the endpoints $(1,1),(1,\infty),\text{ and }(\infty,1)$. 
	\end{theorem}
	
	\begin{figure}[H]		
		\centering
		\begin{tikzpicture}[scale=3]
			\fill[lightgray] (0,0)--(1,0)--(1,1)--(0,1)--cycle;
			\draw[thin][->]  (0,0)node[left]{$O$} --(1.15,0) node[right]{$\frac{1}{p_1}$};
			\draw[thin][->]  (0,0) --(0,1.2) node[left]{$\frac{1}{p_2}$};
			\draw [thin] (1,0) node[below]{$A$} --(1,1) node[right]{$(1,1)$}--(0,1)node[left]{$B$};       
			\node at (1,0) {\tiny{$\circ$}};
			\node at (0,1) {\tiny{$\circ$}};       
			\node at (1,1) {\tiny{$\circ$}};
		\end{tikzpicture}
		\caption{Boundedness region for the operator $\mathcal M_{lac}$ in dimension  $d\geq4$.}
	\end{figure}
	Moreover, due to Stein's maximal principle, we have the following result about pointwise convergence of bilinear spherical averages over an arbitrary set of dilations.
	\begin{corollary}
		Let $\{t_j\}$ be a sequence of positive real numbers such that $t_j\rightarrow 0$. Let $E=\{t_j:j\in\N\}$ have the upper Minkowski dimension $\beta$. Then for $f_1\in L^{p_1}(\R^d)$ and $f_2\in L^{p_2}(\R^d)$, where  $\big(\frac{1}{p_1},\frac{1}{p_2}\big)$ lies in the region given in \Cref{main}, the following pointwise almost everywhere convergence holds
		\[\lim_{j\to\infty}\mathcal{A}_{t_j}(f_1,f_2)(x)=f_1(x)f_2(x).\]
	\end{corollary}
	\subsection*{Methodology of proofs}
	As indicated earlier, the slicing method does not exploit the geometry of a general dilation set and hence does not improve the range of $L^p-$estimates for $\mathcal M_{E}$ beyond that of $\mathcal M_{full}$. We overcome this difficulty by adopting the following strategy to prove \Cref{main}.  
	\begin{itemize}
		\item First, we perform Littlewood--Paley decomposition of both the functions in frequency space. The low--low and mixed-frequency terms are controlled directly by a product of Hardy--Littlewood maximal functions; see \Cref{trivialestimates}. This step reduces the problem to proving the desired bounds for terms with high frequencies. 
		\item Next, we consider the local maximal operator associated with the set $E_0=E\cap [1,2]$ and carry out the decomposition of multiplier as used by Heo, Hong and Yang~ \cite{HHY}. We derive
		\[L^2\times L^2\to L^1,\quad L^2\times L^1\to L^{2/3}, \quad \text{and}~L^1\times L^2\to L^{2/3}\]
		estimates, see \Cref{l2} and \Cref{l2/3}, with suitable decay for the operator norm depending on the frequency parameter and the upper Minkowski dimension of $E_0$. A spatial localization argument then gives the endpoint estimate $L^1\times L^1\to L^{1/2}$ in \Cref{l1}. This endpoint estimate has growth in terms of the frequency parameter and the number of intervals used in discretization of $E_0$. 
		\item The interpolation argument between the growth estimate at endpoints and the decay estimate at points mentioned above yields a region of $L^p-$boundedness of $\mathcal{M}_{E_0}$ depending on upper Minkowski dimension of $E_0$.
		\item Finally, we extend the results obtained for local maximal operator to the global operators associated with $E\subset\R_+$. This part is done by using appropriate rescaling arguments and an adaptation of the bilinear Calderón-Zygmund theory to the underlying situation. We exploit the methods used in \cite{ChristZhou} and \cite{BCSS2}.
	\end{itemize}	
	\subsection*{Organization of the paper} In \Cref{sec:reductionlp}, we carry out the Littlewood--Paley decomposition of both the functions. In \Cref{multiplier}, we decompose the multiplier corresponding to the bilinear spherical average. In \Cref{local}, we obtain $L^p-$estimates for local bilinear spherical maximal function. In \Cref{full} and \Cref{d=1}, we prove the main results \Cref{main} and \Cref{main2} respectively.
	\section{Reduction using Littlewood--Paley decomposition}\label{sec:reductionlp}	
	Consider the following partition of the identity 
	\begin{align*}\label{identity}
		\widehat \phi(\xi)+\sum_{j=1}^\infty\widehat \psi_{2^{-j}}(\xi)=1,\quad\xi\neq 0,
	\end{align*}
	where $\phi\in \mathcal S(\R^d)$ is a radial function such that $\text{supp}(\widehat\phi)\subseteq B(0,2)$ and $\widehat\phi(\xi)=1$ for $\xi\in B(0,1)$. We have used the notation   $\widehat{\phi_t}(\xi)=\widehat\phi(t\xi)$ and 
	$\widehat{\psi}_t(\xi)= \widehat{\phi}(t\xi)-\widehat{\phi}(2t\xi)$. 
	
	For $i\in\Z$, we define smooth Fourier projection operators $\widehat{P_if}(\xi)=\hat{f}(\xi)\hat{\phi}(2^{-i}\xi)$ and $\widehat{R_{i}f}(\xi)=\hat{f}(\xi)\hat{\psi}(2^{-i}\xi)$. Using this, we can write  
	\[f=P_if+\sum_{j=1}^\infty R_{i+j}f.\]
	Consequently, we get the following decomposition of the operator $\mathcal{M}_{E}$.
	\begin{align}\label{deco}
		&\mathcal{M}_{E}(f_1,f_2)(x)\nonumber\\\nonumber	\leq & \sup_{k\in\Z}\sup_{t\in E_k}\Big|\mathcal{A}_{t}(P_kf_1,P_kf_2)(x)\Big|+\sup_{k\in\Z}\sup_{t\in E_k}\bigg|\mathcal{A}_{t}\Big(\sum^{\infty}_{n_1=1}R_{k+n_1}f_1,P_kf_2\Big)(x)\bigg|
		\\\nonumber
		&+\sup_{k\in\Z}\sup_{t\in E_k}\bigg|\mathcal{A}_{t}\Big(P_kf_1,\sum^{\infty}_{n_2=1}R_{k+n_2}f_2\Big)(x)\bigg|+\sup_{k\in\Z}\sup_{t\in E_k}\Big|\mathcal{A}_{t}\Big(\sum^{\infty}_{n_1=1}R_{k+n_1}f_1,\sum^{\infty}_{n_2=1}R_{k+n_2}f_2\Big)(x) \bigg|\\\nonumber
		\lesssim & \sup_{k\in\Z}\sup_{t\in E_k}\Big|\mathcal{A}_{t}(P_kf_1,P_kf_2)(x)\Big|+\sup_{k\in\Z}\sup_{t\in E_k}\Big|\mathcal{A}_{t}(P_kf_1,f_2)(x)\Big|
		\\
		&+\sup_{k\in\Z}\sup_{t\in E_k}\Big|\mathcal{A}_{t}(f_1,P_kf_2)(x)\Big|+\sum^{\infty}_{n_1,n_2=1}\mathcal{M}^{\mathbf{n}}(f_1,f_2)(x),
	\end{align}
	where for $\mathbf{n}=(n_1,n_2)$
	\begin{align*}
		\mathcal{M}^{\mathbf{n}}(f_1,f_2)(x):=\sup_{k\in \Z}\sup_{t\in E_k}\Big|\mathcal{A}_{t}(R_{k+n_1}f_1,R_{k+n_2}f_2)(x)\Big|.
	\end{align*}
	Denote 
	\begin{align*}
		\mathcal{M}^{\mathbf{n}}_k(f_1,f_2)(x)=\sup_{t\in E_k}\Big|\mathcal{A}_{t}(R_{k+n_1}f_1,R_{k+n_2}f_2)(x)\Big|.
	\end{align*}
	The first three terms in \eqref{deco} are controlled by a product of the Hardy-Littlewood maximal functions, which in turn implies strong-type $L^p-$boundedness of $\mathcal{M}_E$ for $p_1,p_2>1$ and weak-type $L^p-$boundedness when $p_i=1$. The strong-type $L^p-$boundedness of the operator $\mathcal{M}_E$ in the case where $p_i=1$ and $1<p_j<\infty$ for $i,j\in\{1,2\},~i\neq j$, is obtained by a standard interpolation argument between weak-type bounds at $\big(1,1,\frac{1}{2}\big)$ and either at $(1,\infty,1)$ or $(\infty,1,1)$.
	\begin{lemma}\label{trivialestimates} Let $d\geq2$. The following pointwise estimates hold
		\begin{align*}
			\sup_{k\in \Z}\sup_{t\in E_k}\Big|\mathcal{A}_{t}(P_kf_1,f_2)(x)\Big|&\lesssim M_{HL}f_1(x)M_{HL}f_2(x),\\
			\sup_{k\in \Z}\sup_{t\in E_k}\Big|\mathcal{A}_{t}(f_1,P_kf_2)(x)\Big|&\lesssim M_{HL}f_1(x)M_{HL}f_2(x),\\
			\sup_{k\in \Z}\sup_{t\in E_k}\Big|\mathcal{A}_{t}(P_kf_1,P_kf_2)(x)\Big|&\lesssim M_{HL}f_{1}(x)M_{HL}f_2(x).
		\end{align*}
	\end{lemma}
	\begin{proof}
		Note that if $t\in E_k$ and $|y|\leq1$, then we have $|ty|\leq2^{-k+1}$. Since $\phi\in\mathcal S(\R^d)$, we get
		\begin{align*}
			\sup_{t\in E_k}\sup_{|y|\leq1}|P_{k}f_1(x-ty)|&\leq\sup_{|ty|\leq2^{-k+1}}\left|\int_{\R^d} f_1(z)2^{kd}\phi(2^k(x-ty-z))~dz\right|\\
			&\lesssim \sup_{|y|\leq2}\left|\int_{\R^d} \frac{2^{kd}f_1(z)}{\Big(4+|2^{k}(x-z)-y|\Big)^{N}}dz\right|\\
			&\lesssim  M_{HL}f_1(x).
		\end{align*}
		Therefore, using the slicing argument, we have
		\begin{align*}
			\sup_{k\in \Z}\sup_{t\in E_k}\Big|\mathcal{A}_{t}(P_{k}f_1,f_2)(x)\Big|
			&\leq \sup_{k\in \Z}\sup_{t\in E_k}\left(\sup_{|y|\leq1}|P_{k}f_1(x-ty)|\int_{\mathbb{S}^{2d-1}}|f_2(x-tz)|d\sigma(y,z)\right)\\
			&\lesssim M_{HL}f_1(x)\sup_{k\in \Z}\sup_{t\in E_k}\int_{B^{d}(0,1)}|f_2(x-tz)|(1-|z|^{2})^{(d-2)/2}\int_{\mathbb{S}^{d-1}}d\sigma(y)~dz\\
			&\lesssim M_{HL}f_1(x)M_{HL}f_2(x).
		\end{align*}
		The remaining two estimates are deduced similarly. 
	\end{proof}
	In view of the lemma above, we are left with proving the desired $L^p-$boundedness of the operator $\mathcal{M}^{\mathbf{n}}$, with a suitable decay in $|\mathbf{n}|$, in order to complete the proof of \Cref{main}.

	\section{Decomposition of multiplier}\label{multiplier}	
	Consider the multiplier form of the bilinear spherical average given by
	\[\mathcal{A}_t(f_1,f_2)(x)=\int_{\R^d\times\R^d}\widehat{f_1}(\xi)\widehat{f_2}(\eta)\widehat{d\sigma}(t\xi,t\eta)e^{ix\cdot(\xi+\eta)}\;d\xi d\eta,\]
	where $\widehat{d\sigma}$ denotes the Fourier transform of the spherical measure and has the expression
	\begin{equation}\label{measure}
		\widehat{d\sigma}(t\xi,t\eta)=C_d\frac{J_{d-1}(|(t\xi,t\eta)|)}{|(t\xi,t\eta)|^{d-1}}.
	\end{equation}
	Here, $J_{d-1}$ denotes the Bessel function of order $d-1$. The following asymptotic estimate of the Bessel function holds 
	\[J_{d-1}(r)\sim r^{-\frac{1}{2}}\bigg(e^{ir}\sum_{j=0}^\infty a_jr^{-j}+e^{-ir}\sum_{j=0}^\infty b_jr^{-j}\bigg)\quad \text{ as }r\to\infty,\]
	for suitable coefficients $a_j$ and $b_j$, such that for all nonnegative integers $M$ and $\alpha$
	\begin{equation}\label{asymptotics}
		\bigg(\frac{d}{dr}\bigg)^\alpha \Bigg[J_{d-1}(r)-r^{-\frac{1}{2}}\bigg(e^{ir}\sum_{j=0}^M a_jr^{-j}+e^{-ir}\sum_{j=0}^M b_jr^{-j}\bigg)\Bigg]=O\big(r^{-\alpha-\frac{M+1}{2}}\big),
	\end{equation}
	as $r\to\infty$. We refer to \cite[pp.~338 and Proposition 3 on pp. 334]{book-Stein} for details.
	
	In view of the asymptotic estimates of Bessel functions, observe that the main term in the multiplier corresponding to the operator $\mathcal A_t$ behaves like $e^{i |t(\xi,\eta)|}|t(\xi,\eta)|^{-\frac{2d-1}{2}}$. The remaining terms satisfy higher order decay, and hence we need to focus on the main term only. 
	Therefore, we prove the desired $L^p-$boundedness of the maximal operator corresponding to bilinear operators defined by 
	\[T^{\mathbf{n}}_{t,k}(f_1,f_2)(x)=\int_{\R^d\times\R^d}m(t\xi,t\eta)\widehat{\psi}\left(2^{-k-n_1} \xi\right)\widehat{\psi}\left(2^{-k-n_2} \eta\right)\widehat{f_1}(\xi)\widehat{f_2}(\eta)e^{ix\cdot(\xi+\eta)}\;d\xi d\eta,\]
	where $t\in E_k$ and $m(\xi,\eta)=e^{i|(\xi,\eta)|}|(\xi,\eta)|^{-\frac{2d-1}{2}}$. 
	
	First, we prove the required estimates for the underlying operators for the case of unit scale, that is $k=0$. The analogous results for the general dyadic scale $k\in\Z$, are derived using a standard scaling argument.
	
	Let us use the notation $T^{\mathbf{n}}_{t}=T^{\mathbf{n}}_{t,0},$ where $t\in E_0$. Due to the lack of additional decay in $|(\xi,\eta)|$ for the derivative of the term $e^{it|(\xi,\eta)|}$ in the multiplier $m$, we need to decompose the exponential term further.
	
	Observe that the support of $\widehat{\psi}$ implies that $2^{n_1-1}\leq|\xi|\leq2^{n_1+1}$ and $2^{n_2-1}\leq|\eta|\leq2^{n_2+1}$. For a given integer $N\in [2^{n_1-1},2^{n_1+1}],$ we write the exponential term as follows 
	\[
	e^{it|(\xi,\eta)|}=e^{it\Phi_N(\xi,\eta)}e^{it\sqrt{N^2+|\eta|^2}},\]
	where
	\[\Phi_N(\xi,\eta):=|(\xi,\eta)|-\sqrt{N^2+|\eta|^2}=\frac{(|\xi|-N)(|\xi|+N)}{\sqrt{|\xi|^2+|\eta|^2}+\sqrt{N^2+|\eta|^2}}.\]
	Observe that the partial derivative of $\Phi_N$ with respect to $\eta$ is given by
	\[\nabla_\eta\Phi_N(\xi,\eta)=\frac{-(|\xi|-N)(|\xi|+N)}{\big(\sqrt{|\xi|^2+|\eta|^2}+\sqrt{N^2+|\eta|^2}\big)^2}\cdot \Big(\frac{\eta}{\sqrt{|\xi|^2+|\eta|^2}}+\frac{\eta}{\sqrt{N^2+|\eta|^2}}\Big).\]
	Therefore, for $\xi$ such that $N\leq|\xi|\leq N+1$, it is easy to verify that the estimate 
	\begin{equation}\label{deri}
		|\partial^\alpha_\eta\Phi_N(\xi,\eta)|\lesssim\min\{1,|\eta|^{-|\alpha|}\},
	\end{equation}
	holds for all multi-indices $\alpha$ uniformly in $\xi$. 
	
	Let $\mathlarger{\chi}_{_N}$ denote the characteristic function of the interval $[N, N+1)$. Define 
	\begin{equation}\label{operatorF}
		F_Nf(x) :=\int_{\mathbb{R}^{d}} \widehat{f}(\xi) \mathlarger{\chi}_{_N}(|\xi|) e^{i x \cdot \xi} d \xi.
	\end{equation}
	This gives us the following decomposition of the operator $T^{\mathbf{n}}_{t}$.
	\begin{align*}
		T^{\mathbf{n}}_{t}(f_1,f_2)(x)&
		=\sum_{N=2^{n_1-1}}^{2^{n_1+1}}T^{\mathbf{n}}_{t}(F_Nf_1,f_2)(x)=:\sum_{N=2^{n_1-1}}^{2^{n_1+1}}S^{\mathbf{n}}_{t,N}(f_1,f_2)(x),
	\end{align*}
	where
	\begin{align}\label{operatorTN}
		S^{\mathbf{n}}_{t,N}(f_1,f_2)(x)&=\int_{\R^d\times\R^d}m_{t,N}^{\mathbf{n}}(\xi,\eta)\widehat{F_Nf_1}(\xi)\widehat{U_{t,N}^{n_2}f_2}(\eta)e^{ix\cdot(\xi+\eta)}\;d\xi d\eta,
	\end{align}
	\begin{equation}\label{operatorU}
		\widehat{U_{t,N}^{n_2}f_2}(\eta):=\widehat{f_2}(\eta) \widehat{\psi}\left(2^{-n_2} \eta\right) e^{it\left(\sqrt{N^{2}+|\eta|^{2}}\right)},
	\end{equation}
	and
	\[m_{t,N}^{\mathbf{n}}(\xi, \eta) :=\frac{e^{it\Phi_N(\xi,\eta)}}{|t(\xi,\eta)|^{\frac{2d-1}{2}}}\widehat{\psi}\left(2^{-n_1} \xi\right) \widehat{\zeta}\left(2^{-n_2} \eta\right).\]
	Here $\zeta\in\mathcal{S}(\R^d)$ such that $\widehat{\zeta}$ is supported in a small neighbourhood of $\supp(\widehat{\psi})$ and $\widehat{\zeta}(\eta)=1$ for $\eta\in\supp(\widehat{\psi})$.
	
	Since $\widehat{U_{t,N}^{n_2}f_2}(\eta)=\int_{\R^d}U_{t,N}^{n_2}f_2(y)e^{-i\eta\cdot y}dy$, we can rewrite
	\begin{align}\label{factor}
		S^{\mathbf{n}}_{t,N}(f_1,f_2)(x)&=\int_{\R^d\times\R^d}\widehat{F_Nf_1}(\xi)U_{t,N}^{n_2}f_2(y)\Big(\int_{\R^d}m_{t,N}^{\mathbf{n}}(\xi, \eta)e^{i\eta\cdot(x-y)}d\eta\Big)e^{ix\cdot\xi}\;d\xi dy.
	\end{align}
	The derivative of $m_{t,N}^{\mathbf{n}}$ is given by
	\begin{align*}
		\partial_\eta \big(m_{t,N}^{\mathbf{n}}(\xi, \eta)\big)&=\frac{e^{it\Phi_N(\xi,\eta)}}{|t(\xi,\eta)|^{\frac{2d-1}{2}}}\widehat{\psi}\left(2^{-n_1} \xi\right) \bigg(\widehat{\zeta}\left(2^{-n_2} \eta\right)\partial_\eta(it\Phi_N(\xi,\eta))-\frac{2d-1}{2}\frac{\widehat{\zeta}\left(2^{-n_2} \eta\right)}{|(\xi,\eta)|}\partial_\eta|(\xi,\eta)|\\
		&\hspace{4.5cm}+2^{-n_2}(\partial\widehat{\zeta})\left(2^{-n_2} \eta\right)\bigg).
	\end{align*}
	By Leibniz rule and the estimate \eqref{deri}, we get that
	\begin{align}\label{derivative}
		|\partial^\alpha_\eta \big(m_{t,N}^{\mathbf{n}}(\xi, \eta)\big)|
		&\lesssim C_{\alpha}2^{-\frac{2d-1}{2}|\mathbf{n}|}2^{-|\alpha|n_2},
	\end{align}    
	for all multi-indices $\alpha$ uniformly in $t\in E_0$. Thus, an integration by parts argument gives us
	\begin{align}\label{ip}
		&\int_{\R^d}m_{t,N}^{\mathbf{n}}(\xi, \eta)e^{i\eta\cdot(x-y)}d\eta\\&
		\nonumber=\frac{1}{(1+|2^{n_2}(x-y)|^2)^{\lceil\frac{d}{2}\rceil+1}}\int_{\R^d}(1-2^{2n_2}\Delta_\eta)^{\lceil\frac{d}{2}\rceil+1}\Big(m_{t,N}^{\mathbf{n}}(\xi, \eta)\Big)e^{i\eta\cdot(x-y)}d\eta,
	\end{align}
	where we have used the fact that
	\[(1-2^{2n_2}\Delta_\eta)e^{i\eta\cdot(x-y)}=(1+|2^{n_2}(x-y)|^2)e^{i\eta\cdot(x-y)}.\]
	Substituting \eqref{ip} into \eqref{factor}, we get
	\begin{align*}
		S^{\mathbf{n}}_{t,N}(f_1,f_2)(x)&=\int_{\mathbb{R}^{d}}A_{n_2}\left(U_{t,N}^{n_2}f_2\right)(x,\eta) B^{\mathbf{n}}_{t,N}\left(F_Nf_1\right)(x, \eta) d \eta,
	\end{align*}
	where
	\begin{align}\label{operatorB}
		B^{\mathbf{n}}_{t,N}\left(F_Nf_1\right)(x, \eta):=&\int_{\mathbb{R}^{d}} \left(1-2^{2n_2} \Delta_{\eta}\right)^{\lceil \frac{d}{2}\rceil+1}\left(m_{t,N}^{\mathbf{n}}(\xi, \eta)\right)\widehat{F_Nf_1}(\xi)e^{ix \cdot \xi} d \xi,
	\end{align}
	and
	\[A_{n_2}\left(U^{n_2}_{t,N}f_2\right)(x,\eta):=\int_{\mathbb{R}^{d}} \frac{U^{n_2}_{t,N}f_2(y)e^{i\eta\cdot(x-y)}}{\left(1+\left|2^{n_2}(x-y)\right|^{2}\right)^{\lceil \frac{d}{2}\rceil+1}} d y.\]
	Note that
	\begin{align*}
		\left|A_{n_2}\left(U^{n_2}_{t,N}f_2\right)(x,\eta)\right|&\leq\int_{\mathbb{R}^{d}} \frac{|U^{n_2}_{t,N}f_2(y)|}{\left(1+\left|2^{n_2}(x-y)\right|^{2}\right)^{\lceil \frac{d}{2}\rceil+1}} d y\\
		&:=L_{n_2}\left(U_{t,N}^{n_2}f_2\right)(x).
	\end{align*}
	This gives us the following bound
	\begin{equation}\label{decomposition}
		|S^{\mathbf{n}}_{t,N}(f_1,f_2)(x)|\leq L_{n_2}\left(U_{t,N}^{n_2}f_2\right)(x) \int_{\mathbb{R}^{d}}\left|B^{\mathbf{n}}_{t,N}\left(F_Nf_1\right)(x, \eta)\right| d \eta.
	\end{equation}
	The following $L^p-$estimates hold for the operators associated with $U_{t,N}^{n_2}$ and $B^{\mathbf{n}}_{t,N}$ introduced as above. For $\mathbf{n}=(n_1,n_2)\in\N^2$, we set $|\mathbf{n}|=\sqrt{n_1^2+n_2^2}$.
	\begin{lemma}\label{bn}
		For $\mathbf{n}\in\N^2$ and $N>0$, let $B^{\mathbf{n}}_{s, N}(f)$ be defined as in \eqref{operatorB}. Then, we have
		\[\int_{\mathbb{R}^{d}} \left\|\sup_{1 \leq s\leq 2}\left|B^{\mathbf{n}}_{s, N}\left(f\right)(\cdot, \eta)\right|\right\|_{L^2} d \eta\lesssim2^{-|\mathbf{n}|\frac{2d-1}{2}}2^{n_2d}\|f\|_{L^2}.\]
	\end{lemma}
	
	\begin{lemma}[\cite{HHY}, Lemma 2.2]\label{u1}
		Let $I\subset[1,2]$. For $n_2\in\N$ and $N>0$, let $U_{s,N}^{n_2}g$ be defined as in \eqref{operatorU}. Then, we have
		\[\left\|\bigg(\int_{I} \big|U_{s,N}^{n_2}g\big|^2 ds\bigg)^{\frac{1}{2}}\right\|_{L^1}\lesssim|I|^{\frac{1}{2}}2^{n_2\frac{d}{2}}\|g\|_{L^1}.\]
	\end{lemma}
	The proof of \Cref{bn} follows by imitating the proof of estimate (2.35) from \cite{HHY} on pages 418-419 along with \eqref{derivative}. \Cref{u1} was proved for $I=[1,2]$ in \cite{HHY}. The same argument can be adapted for any interval $I\subset[1,2]$ to obtain the desired bounds in \Cref{u1}. We skip the proof of \Cref{u1} to avoid repetition. With the decomposition of the multiplier and the results described above, we are in a position to prove the desired $L^p$-estimates of the local maximal function $\mathcal{M}^{\mathbf{n}}_0(f_1,f_2)$ in the following section.
	
	\section{Estimate for local maximal operator $\mathcal M_{E_0}$}\label{local}
	Note that in view of the reduction step using Littlewood--Paley decomposition, see~\Cref{sec:reductionlp} and \Cref{trivialestimates}, we need to prove desired estimates for the maximal operator 
	\[\mathcal{M}^{\mathbf{n}}_0(f_1,f_2)(x)=\sup_{t\in E_0}\Big|\mathcal{A}_{t}(R_{n_1}f_1,R_{n_2}f_2)(x)\Big|.\]
	We have the following results for $\mathcal{M}^{\mathbf{n}}_0$.
	\begin{lemma}\label{l2}
		Let $d\geq1$ and $0\leq \beta\leq1$. Then for any $\mathbf{n}\in\N^2$, we have
		$$\Vert \mathcal{M}^{\mathbf{n}}_0(f_1,f_2)\Vert_{L^1}\lesssim 2^{-|\mathbf{n}|\frac{2d-2-\beta}{2}}\Vert f_1\Vert_{L^2}\Vert f_2\Vert_{L^2}.$$
	\end{lemma}
	\begin{lemma}\label{l2/3}
		Let $d\geq 1$ and $0\leq \beta\leq1$. Then for any $\mathbf{n}\in\N^2$, we have
		\begin{enumerate}
			\item $\left\|\mathcal{M}^{\mathbf{n}}_0(f_1,f_2)\right\|_{L^{2/3}}\lesssim2^{-|\mathbf{n}|\frac{d-3-\beta}{2}}\|f_1\|_{L^1}\|f_2\|_{L^2},$
			\item $\left\|\mathcal{M}^{\mathbf{n}}_0(f_1,f_2)\right\|_{L^{2/3}}\lesssim2^{-|\mathbf{n}|\frac{d-3-\beta}{2}}\|f_1\|_{L^2}\|f_2\|_{L^1}$.
		\end{enumerate}
	\end{lemma}
	
	\begin{lemma}\label{l1}
		Let $d\geq2$ and $0\leq \beta\leq1$. Then for any $\mathbf{n}\in\N^2$, we have
		\begin{equation*}
			\Vert  \mathcal{M}^{\mathbf{n}}_0(f_1,f_2)\Vert_{L^{1/2}}\lesssim 2^{|\mathbf{n}|\beta}\Vert f_1\Vert_{L^1}\Vert f_2\Vert_{L^1}.
		\end{equation*}
	\end{lemma}
	\begin{remark}
		We would like to remark that by an application of Bourgain's interpolation argument \big(\cite[Lemma 2.6]{Lee1}\big) along with the estimates established in \Cref{l2}, \Cref{l2/3}, and \Cref{l1}, we get that for all measurable sets $F_1$ and $F_2$, the following estimate 
		\[\|\mathcal{M}_{E_0}(\chi_{F_1},\chi_{F_2})\|_{L^{p,\infty}} \lesssim |F_1|^{\frac{1}{p_1}}|F_2|^{\frac{1}{p_2}}\]
		holds whenever the point $(\frac{1}{p_1}, \frac{1}{p_2})$ lies on the line segment $PQ$ or $PR$ in dimensions $d \geq 4$. However, observe that for $p<1$, the estimate above does not allow us to deduce $L^{p_1,1}\times L^{p_2,1}\to L^{p,\infty}-$ bounds for the local maximal function $\mathcal{M}_{E_0}(f_1,f_2)$.
	\end{remark}
	Note that if we assume that \Cref{l2}, \Cref{l2/3} and \Cref{l1} hold true, then the bilinear analogue of the real interpolation theorem, see Theorem 7.2.2 in \cite{modernGrafakos}, yields strong-type $L^p-$estimates 
	\begin{align*}
		\Vert \mathcal{M}^{\mathbf{n}}_0(f_1,f_2)\Vert_{L^p}\lesssim 2^{-|\mathbf{n}|\epsilon}\Vert f_1\Vert_{L^{p_1}}\Vert f_2\Vert_{L^{p_2}},
	\end{align*}
	for some $\epsilon>0$ and $\big(\frac{1}{p_1},\frac{1}{p_2}\big)\in\Omega(O,A,Q,P,R,B)\setminus \{A,B,PQ,PR\}$. Therefore, summing over $n_1$ and $n_2,$ we get the desired strong-type $L^p-$estimates for the operator $\mathcal{M}_{E_0}$.

	Therefore, we are left with proving \Cref{l2}, \Cref{l2/3} and \Cref{l1}.
	\subsection{Proof of \Cref{l2}.}\label{Mn0} From \eqref{measure} and the asymptotics of the Bessel function, we have the following decomposition of $\mathcal{M}^{\mathbf{n}}_0$.
	\begin{equation}\label{localdecomposition}
		\mathcal{M}^{\mathbf{n}}_0(f_1,f_2)(x)\leq \mathbb{I}^{\mathbf{n}}(x)+\mathbb{II}^{\mathbf{n}}(x)+\mathbb{III}^{\mathbf{n}}(x),
	\end{equation}
	where
	\[\mathbb{I}^{\mathbf{n}}(x)=\sup_{t\in E_0}\left|\int_{\R^d\times\R^d}\frac{a_0 e^{it|(\xi,\eta)|}}{(t|(\xi,\eta)|)^{\frac{2d-1}{2}}}\widehat{\psi}\left(2^{-n_1} \xi\right)\widehat{\psi}\left(2^{-n_2} \eta\right)\widehat{f_1}(\xi)\widehat{f_2}(\eta)e^{ix\cdot(\xi+\eta)}~d\xi d\eta\right|,\]
	\[\mathbb{II}^{\mathbf{n}}(x)=\sup_{t\in E_0}\left|\int_{\R^d\times\R^d}\frac{b_0 e^{-it|(\xi,\eta)|}}{(t|(\xi,\eta)|)^{\frac{2d-1}{2}}}\widehat{\psi}\left(2^{-n_1} \xi\right)\widehat{\psi}\left(2^{-n_2} \eta\right)\widehat{f_1}(\xi)\widehat{f_2}(\eta)e^{ix\cdot(\xi+\eta)}~d\xi d\eta\right|,\]
	and
	\[\mathbb{III}^{\mathbf{n}}(x)=\sup_{t\in E_0}\left|\int_{\R^d\times\R^d}\mathfrak{m}_{t}^{\mathbf{n}}(\xi,\eta)\widehat{f_1}(\xi)\widehat{f_2}(\eta)e^{ix\cdot(\xi+\eta)}~d\xi d\eta\right|.\]
	Here $\mathfrak{m}_{t}^{\mathbf{n}}(\xi,\eta)=\Big[\widehat{d\sigma}(t\xi,t\eta)-\Big(\frac{a_0 e^{it|(\xi,\eta)|}}{(t|(\xi,\eta)|)^{\frac{2d-1}{2}}}+\frac{b_0 e^{-it|(\xi,\eta)|}}{(t|(\xi,\eta)|)^{\frac{2d-1}{2}}}\Big)\Big]\widehat{\psi}\left(2^{-n_1} \xi\right)\widehat{\psi}\left(2^{-n_2} \eta\right)$. Using the estimate \eqref{asymptotics} with $M=0$, we have
	\[\big|\partial^{\alpha}_{\xi}\partial^{\gamma}_{\eta}\mathfrak{m}_{t}^{\mathbf{n}}(\xi,\eta)\big|\lesssim2^{-|\mathbf{n}|\frac{2d-1}{2}}|\xi|^{-|\alpha|}|\eta|^{-|\gamma|}, \text{ uniformly in }t\in E_0.\]
	Since the multiplier $\mathfrak{m}_{t}^{\mathbf{n}}$ is compactly supported away from origin and satisfies the above derivative condition, we get that
	\begin{equation}\label{error}
		\mathbb{III}^{\mathbf{n}}(x)\lesssim 2^{-|\mathbf{n}|\frac{2d-1}{2}}M_{HL}f_1(x)M_{HL}f_2(x).
	\end{equation}
	Therefore, $\mathbb{III}^{\mathbf{n}}$ satisfies the desired bounds in \Cref{l2}. It remains to prove the estimates in \Cref{l2} for $\mathbb{I}^{\mathbf{n}}$ and $\mathbb{II}^{\mathbf{n}}$. We will establish the bounds for $\mathbb{I}^{\mathbf{n}}$, and the same argument can be applied to $\mathbb{II}^{\mathbf{n}}$ to obtain identical bounds.
	
	Let $\mathcal{E}_0$ be the collection of pairwise disjoint intervals of length $2^{-|\mathbf{n}|}$ covering $E_0$. Then, we can write
	\begin{align*}
		\mathbb{I}^{\mathbf{n}}(x)&=\sup_{t\in E_0}\left|\int_{\R^d\times\R^d}\frac{a_0 e^{it|(\xi,\eta)|}}{(t|(\xi,\eta)|)^{\frac{2d-1}{2}}}\widehat{\psi}\left(2^{-n_1} \xi\right)\widehat{\psi}\left(2^{-n_2} \eta\right)\Bigg(\sum_{N=2^{n_1-1}}^{2^{n_1+1}}\widehat{F_Nf_1}(\xi)\Bigg)\widehat{f_2}(\eta)e^{ix\cdot(\xi+\eta)}~d\xi d\eta\right|\\
		&\leq |a_0|\sum_{N=2^{n_1-1}}^{2^{n_1+1}}\sup_{I\in\mathcal{E}_0}\sup_{t\in I}\Bigg|S^{\mathbf{n}}_{t,N}(f_1,f_2)(x) \Bigg|,
	\end{align*}
	where the operators $F_N$ and $S^{\mathbf{n}}_{t,N}$ are defined in \eqref{operatorF} and \eqref{operatorTN} respectively.
	
	For $t\in I$, using the Fundamental Theorem of Calculus, we get that 
	\begin{align}\label{FTC}
		&\sup_{t\in I}\Big|S^{\mathbf{n}}_{t,N}(f_1,f_2)(x) \Big|\\
		\nonumber\lesssim&\left(\frac{1}{|I|}\int_I|S^{\mathbf{n}}_{s,N}(f_1,f_2)(x)|^2\;ds\right)^\frac{1}{2}+\left(\int_I\big|S^{\mathbf{n}}_{s,N}(f_1,f_2)(x)\frac{\partial}{\partial s}S^{\mathbf{n}}_{s,N}(f_1,f_2)(x)\big|\;ds\right)^\frac{1}{2}\\
		\nonumber\leq& \left(\frac{1}{|I|}\int_I|S^{\mathbf{n}}_{s,N}(f_1,f_2)(x)|^2\;ds\right)^\frac{1}{2}+\left(\int_I|S^{\mathbf{n}}_{s,N}(f_1,f_2)(x)|^2\;ds\right)^\frac{1}{4}\left(\int_I\left|\frac{\partial}{\partial s}S^{\mathbf{n}}_{s,N}(f_1,f_2)(x)\right|^2\;ds\right)^\frac{1}{4}.
	\end{align}
	Using $\ell^2\hookrightarrow\ell^\infty$ embedding, we get that 
	\begin{align*}\label{localdeco}
		\mathbb{I}^{\mathbf{n}}(x)\lesssim&\sum_{N=2^{n_1-1}}^{2^{n_1+1}}2^{\frac{|\mathbf{n}|}{2}}\mathcal{G}^{\mathbf{n}}_{0,N}(f_1,f_2)(x)+\sum_{N=2^{n_1-1}}^{2^{n_1+1}}\big(\mathcal{G}^{\mathbf{n}}_{0,N}(f_1,f_2)(x)\tilde{\mathcal{G}}^{\mathbf{n}}_{0,N}(f_1,f_2)(x)\big)^{\frac{1}{2}},
	\end{align*}
	where
	\[\mathcal{G}^{\mathbf{n}}_{0,N}(f_1,f_2)(x):=\left(\sum_{I\in \mathcal{E}_0}\int_I|S^{\mathbf{n}}_{s,N}(f_1,f_2)(x)|^2\;ds\right)^\frac{1}{2}~\text{ and}\]
	\[\tilde{\mathcal{G}}^{\mathbf{n}}_{0,N}(f_1,f_2)(x):=\left(\sum_{I\in \mathcal{E}_0}\int_I\left|\frac{\partial}{\partial s}S^{\mathbf{n}}_{s,N}(f_1,f_2)(x)\right|^2\;ds\right)^\frac{1}{2}.\]    
	Thus, we have 
	\begin{align}\label{MtoG}
		\left\|\mathbb{I}^{\mathbf{n}}\right\|_{L^p}\lesssim2^{\frac{|\mathbf{n}|}{2}}\Bigg\|\sum_{N=2^{n_1-1}}^{2^{n_1+1}}\mathcal{G}^{\mathbf{n}}_{0,N}(f_1,f_2)\Bigg\|_{L^p}+\Bigg\|\sum_{N=2^{n_1-1}}^{2^{n_1+1}}\mathcal{G}^{\mathbf{n}}_{0,N}(f_1,f_2)\Bigg\|_{L^p}^\frac{1}{2}\Bigg\|\sum_{N=2^{n_1-1}}^{2^{n_1+1}}\tilde{\mathcal{G}}^{\mathbf{n}}_{0,N}(f_1,f_2)\Bigg\|_{L^p}^\frac{1}{2}.
	\end{align}
	Since the multiplier $2^{-|\mathbf{n}|}\frac{\partial}{\partial s}\frac{ e^{is|(\xi,\eta)|}}{(s|(\xi,\eta)|)^{\frac{2d-1}{2}}}$ behaves like $\frac{ e^{is|(\xi,\eta)|}}{(s|(\xi,\eta)|)^{\frac{2d-1}{2}}}$ for $1\leq s\leq2$, it is enough to prove $L^p-$estimates for the square function $\Big(\int_I|S^{\mathbf{n}}_{s,N}(f_1,f_2)(x)|^2\;ds\Big)^\frac{1}{2}$. We invoke the decomposition of the operator $S^{\mathbf{n}}_{t,N}$ from \eqref{decomposition} to get that 
	\begin{align*}
		\mathcal{G}^{\mathbf{n}}_{0,N}(f_1,f_2)(x)\lesssim\left(\int_{\mathbb{R}^{d}} \sup_{1 \leq s\leq 2}\left|B^{\mathbf{n}}_{s, N}\left(F_{N}f_1\right)(x, \eta)\right| d \eta\right)\bigg(\sum_{I\in \mathcal{E}_0}\int_I\left|L_{n_2}\left(U_{s,N}^{n_2}f_2\right)(x)\right|^2\;ds\bigg)^\frac{1}{2}.
	\end{align*}
	Therefore, it is enough to prove $L^p-$estimates for the following operator 
	\[f_1\rightarrow \int_{\mathbb{R}^{d}} \sup_{1 \leq s\leq 2}\left|B^{\mathbf{n}}_{s, N}\left(F_{N}f_1\right)(x, \eta)\right| d \eta\] 
	and 
	\[f_2\rightarrow \bigg(\sum_{I\in \mathcal{E}_0}\int_I\left|L_{n_2}\left(U_{s,N}^{n_2}f_2\right)(x)\right|^2\;ds\bigg)^\frac{1}{2}.\] 
	Observe that using Minkowski's integral inequality, we have
	\begin{align}\label{Lneq}
		\bigg(\int_I\left|L_{n_2}\left(U_{s,N}^{n_2}f_2\right)(x)\right|^2\;ds\bigg)^\frac{1}{2}\leq&\int_{\R^d}\bigg(\int_I\left|U_{s,N}^{n_2}f_2(x-y)\right|^2\;ds\bigg)^\frac{1}{2}\frac{dy}{\left(1+\left|2^{n_2}y\right|^{2}\right)^{\lceil \frac{d}{2}\rceil+1}}.
	\end{align}
	We take $L^2-$norm in the above and use Plancherel's theorem to get that 
	\begin{align}\label{U2}
		\left\|\bigg(\sum_{I\in \mathcal{E}_0}\int_I\left|L_{n_2}\left(U_{s,N}^{n_2}f_2\right)\right|^2\;ds\bigg)^\frac{1}{2}\right\|_{L^2}\leq&2^{|\mathbf{n}|\frac{\beta}{2}}\left\|\bigg(\int_I\big|U_{s,N}^{n_2}f_2\big|^2\;ds\bigg)^\frac{1}{2}\right\|_{L^2}\int_{\R^d}\frac{dy}{\left(1+\left|2^{n_2}y\right|^{2}\right)^{\lceil \frac{d}{2}\rceil+1}}\\
		\lesssim&\nonumber2^{-n_2d}2^{-|\mathbf{n}|\frac{1-\beta}{2}}\|f_2\|_{L^2}.
	\end{align}	
	Next, by Minkowski's integral inequality and an application of \Cref{bn}, we obtain
	\begin{equation}\label{Bn}
		\left\|\int_{\mathbb{R}^{d}} \sup_{1 \leq s\leq 2}\left|B^{\mathbf{n}}_{s, N}\left(F_{N}f_1\right)(\cdot, \eta)\right| d \eta\right\|_{L^2}\lesssim2^{-\frac{2d-1}{2}|\mathbf{n}|}2^{n_2d}\|F_Nf_1\|_{L^2}.
	\end{equation}
	Using Cauchy-Schwarz inequality along with the estimates \eqref{U2} and \eqref{Bn}, we obtain
	\begin{align*}
		\Bigg\|\sum_{N=2^{n_1-1}}^{2^{n_1+1}}\mathcal{G}^{\mathbf{n}}_{0,N}(f_1,f_2)\Bigg\|_{L^1}\lesssim&2^{-|\mathbf{n}|\frac{2d-\beta}{2}}\sum_{N=2^{n_1-1}}^{2^{n_1+1}}\|F_Nf_1\|_{L^2}\|f_2\|_{L^2}\\
		\lesssim&2^{-|\mathbf{n}|\frac{2d-1-\beta}{2}}\|f_1\|_{L^2}\|f_2\|_{L^2}.
	\end{align*}
	
	As described earlier, similar arguments as above, yield the following bounds for the square function involving the derivative term $\frac{\partial}{\partial s}S^{\mathbf{n}}_{t,N}$.
	\begin{align*}
		\Bigg\|\sum_{N=2^{n_1-1}}^{2^{n_1+1}}\tilde{\mathcal{G}}^{\mathbf{n}}_{0,N}(f_1,f_2)\Bigg\|_{L^1}\lesssim2^{-|\mathbf{n}|\frac{2d-3-\beta}{2}}\|f_1\|_{L^2}\|f_2\|_{L^2}.
	\end{align*}
	Substituting the estimates obtained above in \eqref{MtoG}, we have
	\begin{align*}
		\left\|\mathbb{I}^{\mathbf{n}}\right\|_{L^1}\lesssim& 2^{\frac{|\mathbf{n}|}{2}}2^{-|\mathbf{n}|\frac{2d-1-\beta}{2}}\|f_1\|_{L^2}\|f_2\|_{L^2}+\left(2^{-|\mathbf{n}|\frac{2d-1-\beta}{2}}\|f_1\|_{L^2}\|f_2\|_{L^2}\right)^{\frac{1}{2}}\left(2^{-|\mathbf{n}|\frac{2d-3-\beta}{2}}\|f_1\|_{L^2}\|f_2\|_{L^2}\right)^{\frac{1}{2}}\\
		\lesssim &2^{-|\mathbf{n}|\frac{2d-2-\beta}{2}}\|f_1\|_{L^2}\|f_2\|_{L^2}.
	\end{align*}
	This proves the desired estimates for $\mathbb{I}^{\mathbf{n}}$, and similarly for $\mathbb{II}^{\mathbf{n}}$. These estimates, together with \eqref{error}, complete the proof of \Cref{l2}. \qed
	
	\subsection{Proof of \Cref{l2/3}.}
	Based on the decomposition of $\mathcal{M}^{\mathbf{n}}_0$ given in \eqref{localdecomposition} and the estimates for $\mathbb{III}^{\mathbf{n}}$ in \eqref{error}, it is enough to prove the desired estimates in \Cref{l2/3} for $\mathbb{I}^{\mathbf{n}}$ and $\mathbb{II}^{\mathbf{n}}$. Again, we prove the bounds for $\mathbb{I}^{\mathbf{n}}$, and the same bounds follow for $\mathbb{II}^{\mathbf{n}}$ by a similar argument.
	
	The estimate from \eqref{Lneq} along with \Cref{u1} gives us 
	\begin{align}\label{U1}
		\left\|\bigg(\sum_{I\in \mathcal{E}_0}\int_I\left|L_{n_2}\left(U_{s,N}^{n_2}f_2\right)\right|^2\;ds\bigg)^\frac{1}{2}\right\|_{L^1}\lesssim2^{-n_2\frac{d}{2}}2^{-|\mathbf{n}|\frac{1-\beta}{2}}\|f_2\|_{L^1}.
	\end{align}

	We first record a measurable-set version of \Cref{u1}.
	
	\begin{lemma}\label{u1-measurable}
		Let \(A\subset[1,2]\) be measurable, let \(n_2\in\mathbb N\),
		and let \(N>0\). Then
		\begin{equation}\label{u1-set}
			\left\|
			\left(
			\int_A |U_{s,N}^{n_2}g|^2\,ds
			\right)^{1/2}
			\right\|_{L^1(\mathbb R^d)}
			\lesssim
			|A|^{1/2}2^{n_2d/2}\|g\|_{L^1(\mathbb R^d)}.
		\end{equation}
		The implicit constant is independent of \(A\), \(N\), and \(n_2\).
	\end{lemma}
	
	\begin{proof}
		Let
		\[
		m_{s,N}^{n_2}(\eta)
		:=
		\widehat\psi(2^{-n_2}\eta)
		e^{is\sqrt{N^2+|\eta|^2}}
		\]
		and denote its inverse Fourier transform by
		\[
		\mathcal K_{s,N}^{n_2}(x)
		:=
		\mathcal F^{-1}\bigl(m_{s,N}^{n_2}\bigr)(x).
		\]
		Thus
		\[
		U_{s,N}^{n_2}g
		=
		g*\mathcal K_{s,N}^{n_2}.
		\]
		
		Put
		\[
		M:=\left\lceil\frac d2\right\rceil+1.
		\]
		Since \(\widehat\psi(2^{-n_2}\eta)\) is supported where
		\(|\eta|\simeq 2^{n_2}\), we have, for every multi-index
		\(\alpha\) with \(|\alpha|\leq M\),
		\[
		\bigl|
		\partial_\eta^\alpha m_{s,N}^{n_2}(\eta)
		\bigr|
		\lesssim_\alpha
		\mathbf 1_{\{c2^{n_2}\leq|\eta|\leq C2^{n_2}\}}(\eta),
		\qquad s\in[1,2].
		\]
		Indeed, on the support of \(\widehat\psi(2^{-n_2}\cdot)\),
		\[
		\left|
		\partial_\eta^\gamma
		\sqrt{N^2+|\eta|^2}
		\right|
		\lesssim_\gamma
		2^{n_2(1-|\gamma|)},
		\qquad |\gamma|\geq1,
		\]
		uniformly in \(N>0\). Hence derivatives of the oscillatory
		factor are uniformly bounded, while derivatives of
		\(\widehat\psi(2^{-n_2}\eta)\) produce nonpositive powers of
		\(2^{n_2}\). It follows that
		\[
		\left\|
		\partial_\eta^\alpha m_{s,N}^{n_2}
		\right\|_{L^2_\eta}
		\lesssim_\alpha
		2^{n_2d/2}.
		\]
		
		Weighted Plancherel's theorem therefore gives
		\begin{align}
			\int_{\mathbb R^d}
			(1+|x|^2)^M
			|\mathcal K_{s,N}^{n_2}(x)|^2\,dx
			&\lesssim
			\sum_{|\alpha|\leq M}
			\left\|
			\partial_\eta^\alpha m_{s,N}^{n_2}
			\right\|_{L^2_\eta}^2
			\nonumber\\
			&\lesssim
			2^{n_2d},
			\label{weighted-kernel}
		\end{align}
		uniformly for \(s\in[1,2]\).
		
		Define
		\[
		\mathfrak K_A(x)
		:=
		\left(
		\int_A
		|\mathcal K_{s,N}^{n_2}(x)|^2\,ds
		\right)^{1/2}.
		\]
		Since \(M>d/2\), the function \((1+|x|^2)^{-M}\) is
		integrable. Thus, by the Cauchy--Schwarz inequality,
		Tonelli's theorem, and \eqref{weighted-kernel},
		\begin{align*}
			\|\mathfrak K_A\|_{L^1}
			&\leq
			\left(
			\int_{\mathbb R^d}(1+|x|^2)^{-M}\,dx
			\right)^{1/2}
			\left(
			\int_{\mathbb R^d}
			(1+|x|^2)^M
			\int_A
			|\mathcal K_{s,N}^{n_2}(x)|^2\,ds\,dx
			\right)^{1/2}
			\\
			&=
			\left(
			\int_{\mathbb R^d}(1+|x|^2)^{-M}\,dx
			\right)^{1/2}
			\left(
			\int_A
			\int_{\mathbb R^d}
			(1+|x|^2)^M
			|\mathcal K_{s,N}^{n_2}(x)|^2\,dx\,ds
			\right)^{1/2}
			\\
			&\lesssim
			|A|^{1/2}2^{n_2d/2}.
		\end{align*}
		
		Finally, Minkowski's integral inequality gives
		\begin{align*}
			\left(
			\int_A
			|U_{s,N}^{n_2}g(x)|^2\,ds
			\right)^{1/2}
			&=
			\left(
			\int_A
			\left|
			\int_{\mathbb R^d}
			g(z)\mathcal K_{s,N}^{n_2}(x-z)\,dz
			\right|^2
			ds
			\right)^{1/2}
			\\
			&\leq
			\int_{\mathbb R^d}
			|g(z)|
			\left(
			\int_A
			|\mathcal K_{s,N}^{n_2}(x-z)|^2\,ds
			\right)^{1/2}
			dz
			\\
			&=
			|g|*\mathfrak K_A(x).
		\end{align*}
		Taking the \(L^1\) norm and applying Young's convolution
		inequality, we obtain
		\[
		\left\|
		\left(
		\int_A |U_{s,N}^{n_2}g|^2\,ds
		\right)^{1/2}
		\right\|_{L^1}
		\leq
		\|g\|_{L^1}\|\mathfrak K_A\|_{L^1}
		\lesssim
		|A|^{1/2}2^{n_2d/2}\|g\|_{L^1}.
		\]
		This proves \eqref{u1-set}.
	\end{proof}

	Now, applying H\"older's inequality along with the estimates \eqref{Bn} and \eqref{U1}, we obtain
	\begin{align*}
		\Bigg\|\sum_{N=2^{n_1-1}}^{2^{n_1+1}}\mathcal{G}^{\mathbf{n}}_{0,N}(f_1,f_2)\Bigg\|_{L^{2/3}}&\lesssim2^{-|\mathbf{n}|\frac{d-\beta}{2}}\bigg(\sum_{N=2^{n_1-1}}^{2^{n_1+1}}\|F_Nf_1\|_{L^2}^{\frac{2}{3}}\|f_2\|_{L^1}^{\frac{2}{3}}\bigg)^{\frac{3}{2}}\\
		&\lesssim2^{-|\mathbf{n}|\frac{d-2-\beta}{2}}\|f_1\|_{L^2}\|f_2\|_{L^1},
	\end{align*}
	and
	\begin{align*}
		\Bigg\|\sum_{N=2^{n_1-1}}^{2^{n_1+1}}\tilde{\mathcal{G}}^{\mathbf{n}}_{0,N}(f_1,f_2)\Bigg\|_{L^{2/3}}&\lesssim2^{-|\mathbf{n}|\frac{d-4-\beta}{2}}\|f_1\|_{L^2}\|f_2\|_{L^1}. 
	\end{align*}
	Combining the estimates for square functions together with equation \eqref{MtoG}, we obtain 
	\begin{align}\label{In1}
		\left\|\mathbb{I}^{\mathbf{n}}\right\|_{L^{2/3}}\lesssim2^{-|\mathbf{n}|\frac{d-3-\beta}{2}}\|f_1\|_{L^2}\|f_2\|_{L^1}.
	\end{align}
	Due to symmetry in $\xi$ and $\eta$ variables, we also have 
	\begin{align*}
		\left\|\mathbb{I}^{\mathbf{n}}\right\|_{L^{2/3}}&\lesssim2^{-|\mathbf{n}|\frac{d-3-\beta}{2}}\|f_1\|_{L^1}\|f_2\|_{L^2}.
	\end{align*}
	Combining the above estimates for $\mathbb{I}^{\mathbf{n}}$ and similar estimates for $\mathbb{II}^{\mathbf{n}}$ along with \eqref{error} completes the proof of \Cref{l2/3}.
	\qed	
	\subsection{Proof of \Cref{l1}.}
	Let $\Psi_{2^{-i}}(w)=2^{id}(2+2^i|w|)^{-(d+1)}$ so that $|\psi_{2^{-i}}(w)|\lesssim\Psi_{2^{-i}}(w)$. We note that $\Psi_{2^{-i}}(w_1)\lesssim\Psi_{2^{-i}}(w_2)$ whenever $|w_1-w_2|\leq2^{-i}$. Indeed, when $|w_1-w_2|<|w_2|$, we have
	\begin{align*}
		\Psi_{2^{-i}}(w_1)&=\Psi_{2^{-i}}(w_1-w_2+w_2)\\
		&=2^{id}(2+2^{i}|w_1-w_2+w_2|)^{-(d+1)}\\
		&\leq 2^{id}(2+2^{i}(|w_2|-|w_1-w_2|))^{-(d+1)}\\
		&\leq 2^{id}(1+2^{i}|w_2|)^{-(d+1)}\lesssim \Psi_{2^{-i}}(w_2).
	\end{align*}
	On the other hand, when $|w_2|\leq|w_1-w_2|\leq2^{-i}$, we have
	\begin{align*}
		\Psi_{2^{-i}}(w_1)&\leq 2^{id}(2+2^{i}(|w_1-w_2|-|w_2|))^{-(d+1)}\\
		&\leq 2^{id}\\
		&\lesssim 2^{id}(2+2^{i}|w_2|)^{-(d+1)}= \Psi_{2^{-i}}(w_2).
	\end{align*}
	Next, we consider the partition $\R^d=\cup_{Q\in\mathfrak C}Q$, where $Q\in\mathfrak{C}$ are cubes with unit side length and sides parallel to coordinate axes. For $a>0,$ let $aQ$ denote the cube concentric with $Q$ having side length $a\,\ell(Q)$. With this, we can write 
	\begin{align*}
		&\|\mathcal{M}^{\mathbf{n}}_0(f_1,f_2)\|_{L^{1/2}}^\frac{1}{2}\\	
		&=\int_{\R^d}\Bigg(\sup\limits_{t\in E_0}\bigg|\int_{\mathbb{S}^{2d-1}}\bigg(\sum\limits_{Q\in\mathfrak{C}}(f_1*\psi_{2^{-n_1}})\mathlarger{\chi}_{_Q}\bigg)(x-ty)( f_2*\psi_{2^{-n_2}})(x-tz)\;d\sigma(y,z)\bigg|\Bigg)^\frac{1}{2}dx\\
		&\leq\sum_{Q\in\mathfrak{C}}\int_{5Q}\bigg(\sup\limits_{t\in E_0}\int_{\mathbb{S}^{2d-1}}\big((|f_1|*\Psi_{2^{-n_1}})\mathlarger{\chi}_{_Q}\big)(x-ty)\big(( |f_2|*\Psi_{2^{-n_2}})\mathlarger{\chi}_{_{10Q}}\big)(x-tz)\;d\sigma(y,z)\bigg)^\frac{1}{2}dx.
	\end{align*}
	Let $I\in \mathcal{E}_0$ be any interval. Without loss of generality, we may assume  that $I=[t_0-2^{-|\mathbf{n}|-1},t_0+2^{-|\mathbf{n}|-1}]$, where $t_0\in [1,2]$. For any $t\in I$ and $|y|\leq1$, we have
	\begin{align*}
		|f_1|*\Psi_{2^{-n_1}}(x-ty)&=|f_1|*\Psi_{2^{-n_1}}(x-t_0y+(t_0-t)y)\\
		&=\int_{\R^d}|f_1|(z)\Psi_{2^{-n_1}}(x-t_0y+(t_0-t)y-z)\;dz\\
		&\lesssim \int_{\R^d}|f_1|(z)\Psi_{2^{-n_1}}(x-t_0y-z)\;dz\\
		&=|f_1|*\Psi_{2^{-n_1}}(x-t_0y),
	\end{align*}
	where we have used the fact that $\Psi_{2^{-n_1}}(w_1)\lesssim\Psi_{2^{-n_1}}(w_2)$ when $|w_1-w_2|\leq2^{-|\mathbf{n}|}\leq 2^{-n_1}$ in the last inequality. Similarly, for any $|z|\leq1$ and $t\in I$, we have 
	\[|f_2|*\Psi_{2^{-n_2}}(x-tz)\lesssim|f_2|*\Psi_{2^{-n_2}}(x-t_0z).\]
	Using H\"older's inequality and the estimates above, we have
	\begin{align*}
		&\|\mathcal{M}^{\mathbf{n}}_0(f_1,f_2)\|_{L^{1/2}}^\frac{1}{2}\\
		&\lesssim\sum_{Q\in\mathfrak{C}}\bigg(\int_{5Q}\sum_{I\in\mathcal{E}_0}\sup_{t\in I}\int_{\mathbb S^{2d-1}}\left((|f_1|*\Psi_{2^{-n_1}})\mathlarger{\chi}_{_{Q}}\right)(x-ty)\left(( |f_2|*\Psi_{2^{-n_2}})\mathlarger{\chi}_{_{10Q}}\right)(x-tz)\;d\sigma(y,z)\;dx\bigg)^\frac{1}{2}\\
		&\lesssim\sum_{Q\in\mathfrak{C}}\bigg(\int_{5Q}\sum_{I\in\mathcal{E}_0}\int_{\mathbb S^{2d-1}}\left((|f_1|*\Psi_{2^{-n_1}})\mathlarger{\chi}_{_{2Q}}\right)(x-t_0y)\left(( |f_2|*\Psi_{2^{-n_2}})\mathlarger{\chi}_{_{6Q}}\right)(x-t_0z)\;d\sigma(y,z)\;dx\bigg)^\frac{1}{2}\\
		&\lesssim\sum_{Q\in\mathfrak{C}}\bigg(\sum_{I\in\mathcal{E}_0}\big\|(|f_1|*\Psi_{2^{-n_1}})\mathlarger{\chi}_{_{2Q}}\big\|_{L^1}\;\big\|( |f_2|*\Psi_{2^{-n_2}})\mathlarger{\chi}_{_{6Q}}\big\|_{L^1}\bigg)^\frac{1}{2}\\
		&\leq2^{|\mathbf{n}|\frac{\beta}{2}}\bigg(\sum_{Q\in\mathfrak{C}}\|(|f_1|*\Psi_{2^{-n_1}})\mathlarger{\chi}_{_{2Q}}\|_{L^1}\bigg)^\frac{1}{2}\bigg(\sum_{Q\in\mathfrak{C}}\|(|f_2|*\Psi_{2^{-n_2}})\mathlarger{\chi}_{_{6Q}}\|_{L^1}\bigg)^\frac{1}{2}\\
		&\lesssim2^{|\mathbf{n}|\frac{\beta}{2}}\|f_1\|_{L^1}^\frac{1}{2}\|f_2\|_{L^1}^\frac{1}{2},
	\end{align*}
	where we have used the $L^1\times L^1\to L^1-$boundedness of the single bilinear spherical average $\mathcal A_t$ obtained in \cite{IPS} in the third step and Cauchy-Schwarz inequality in the fourth step of the equation above. This completes the proof of \Cref{l1}.
	\qed
	
	\section{Estimate for global maximal operator: Proof of \Cref{main}}\label{full}
	Recall that in order to complete the proof of \Cref{main} we are required to prove the desired $L^p-$estimates for the operator 
	\[\mathcal{M}^{\mathbf{n}}(f_1,f_2)(x)=\sup_{k\in \Z}\mathcal{M}^{\mathbf{n}}_k(f_1,f_2)(x).\]
	Note that by a change of variable $\xi\to 2^k\xi$ and $\eta\to2^k\eta$, we have
	\begin{align}\label{scaling}
		\mathcal{M}^{\mathbf{n}}_k(f_1,f_2)(x)=&\sup_{t\in E_k}\left|\int_{\R^d\times\R^d}\widehat{d\sigma}(t(\xi,\eta))\widehat{\psi}\left(2^{-k-n_1} \xi\right)\widehat{\psi}\left(2^{-k-n_2} \eta\right)\widehat{f_1}(\xi)\widehat{f_2}(\eta)e^{ix\cdot(\xi+\eta)}~d\xi d\eta\right|\\
		=&\nonumber\sup_{s\in J}\left|\int_{\R^d\times\R^d}\widehat{d\sigma}(s(\xi,\eta))\widehat{\psi}\left(2^{-n_1} \xi\right)\widehat{\psi}\left(2^{-n_2} \eta\right)\widehat{(f_1)_{2^{-k}}}(\xi)\widehat{(f_2)_{2^{-k}}}(\eta)e^{i2^kx\cdot(\xi+\eta)}~d\xi d\eta\right|,
	\end{align}
	where $s=2^kt\in 2^kE_k=J$ with $J\subset[1,2]$ and $(f)_{2^{-k}}=f(2^{-k}x)$.

	Since the notion of Minkowski dimension as described in \Cref{def} is uniform in $k$, the operator $\mathcal{M}^{\mathbf{n}}_k$ satisfies the same bounds as those of $\mathcal{M}^{\mathbf{n}}_0$ proved in \Cref{l2}, \Cref{l2/3} and \Cref{l1} for H\"{o}lder exponents $(p_1,p_2,p)$. Next, we use these estimates for each fixed scale to prove the following results for the operator $\mathcal{M}^{\mathbf{n}}$. 
	\begin{lemma}\label{L2}
		Let $d\geq2$ and $0\leq \beta\leq1$. Then for any $\mathbf{n}\in\N^2$, we have
		\[\Vert \mathcal{M}^{\mathbf{n}}(f_1,f_2)\Vert_{L^1}\lesssim 2^{-|\mathbf{n}|\frac{2d-2-\beta}{2}}\Vert f_1\Vert_{L^2}\Vert f_2\Vert_{L^2}.\]
	\end{lemma}
	
	\begin{lemma}\label{L2/3}
		Let $d\geq4$ and $0\leq \beta<1$. Then for any $\mathbf{n}\in\N^2$, we have
		\begin{enumerate}
			\item $\left\|\mathcal{M}^{\mathbf{n}}(f_1,f_2)\right\|_{L^{2/3,\infty}}\lesssim |\mathbf{n}|^{3} 2^{-|\mathbf{n}|\frac{d-3-\beta}{2}}\|f_1\|_{L^1}\|f_2\|_{L^2},$
			\item $\left\|\mathcal{M}^{\mathbf{n}}(f_1,f_2)\right\|_{L^{2/3,\infty}}\lesssim |\mathbf{n}|^{3} 2^{-|\mathbf{n}|\frac{d-3-\beta}{2}}\|f_1\|_{L^2}\|f_2\|_{L^1}$.
		\end{enumerate}
	\end{lemma}
	
	\begin{lemma}\label{L1}
		Let $d\geq2$ and $0\leq \beta\leq1$. Then for any $\mathbf{n}\in\N^2$, we have
		\[\Vert\mathcal{M}^{\mathbf{n}}(f_1,f_2)\Vert_{L^{1/2,\infty}}\lesssim |\mathbf{n}|^4 2^{|\mathbf{n}|\beta}\Vert f_1\Vert_{L^1}\Vert f_2\Vert_{L^1}.\]
	\end{lemma}
	Note that if we assume \Cref{L2}, \Cref{L2/3}, and \Cref{L1}, the proof of \Cref{main} follows by arguments similar to those given for the operator $\mathcal{M}_{E_0}$.  
	\subsection{Proof of \Cref{L2}.}
	Recall that $\supp(\widehat{\psi})\subset B(0,2)\setminus B\big(0,\frac{1}{2}\big)$. Let $\boldsymbol{\psi}\in\mathcal S(\R^d)$ such that $\supp(\widehat{\boldsymbol{\psi}})\subset B(0,4)\setminus B\big(0,\frac{1}{4}\big)$ and $\widehat{\boldsymbol{\psi}}(\xi)=1$ for $\xi\in B(0,2)\setminus B\big(0,\frac{1}{2}\big)$. Define
	\[\widehat{\mathbf{R}_{i}f}(\xi)=\hat{f}(\xi)\widehat{\boldsymbol{\psi}}_{2^{-i}}(\xi).\]
	We can rewrite
	\[\mathcal{M}^{\mathbf{n}}_k(f_1,f_2)=\sup_{t\in E_k}\Big|\mathcal{A}_{t}(R_{k+n_1}\circ\mathbf{R}_{k+n_1}f_1,R_{k+n_2}\circ\mathbf{R}_{k+n_2}f_2)(x)\Big|.\]
	Using $\ell_1\hookrightarrow\ell_\infty$ embedding and $L^2\times L^2\to L^1-$estimate from \Cref{l2}, we have
	\begin{align*}
		\left\|\mathcal{M}^{\mathbf{n}}(f_1,f_2)\right\|_{L^1}&\leq \left\|\sum_{k\in\Z}\mathcal{M}^{\mathbf{n}}_k(f_1,f_2)\right\|_{L^1}\\
		&\lesssim \sum_{k\in\Z}2^{-|\mathbf{n}|\frac{2d-2-\beta}{2}}\|\mathbf{R}_{k+n_1}f_1\|_{L^2}\|\mathbf{R}_{k+n_2}f_2\|_{L^2}\\
		&\leq 2^{-|\mathbf{n}|\frac{2d-2-\beta}{2}}\left(\sum_{k\in\Z}\|\mathbf{R}_{k+n_1}f_1\|_{L^2}^2\right)^{\frac{1}{2}}\left(\sum_{k\in\Z}\|\mathbf{R}_{k+n_2}f_2\|_{L^2}^2\right)^{\frac{1}{2}}\\
		&\lesssim 2^{-|\mathbf{n}|\frac{2d-2-\beta}{2}}\|f_1\|_{L^2}\|f_2\|_{L^2}.
	\end{align*}
	This completes the proof of \Cref{L2}.
	\qed
	
	\subsection{Proof of \Cref{L2/3}.}
	Note that due to symmetry, it is enough to prove part $(1)$. The proof is based on a suitable adaptation of the bilinear Calder\'{o}n-Zygmund theory and delicate analysis involving the exceptional set which arises in the Calder\'{o}n-Zygmund decomposition of the input function. Without loss of generality, we may assume that $\|f_1\|_{L^1}=\|f_2\|_{L^2}=1$. Given $\alpha>0$, we shall prove the following
	\begin{equation*}\label{weak2}
		\Big|\{x:\mathcal{M}^{\mathbf{n}}(f_1,f_2)(x)>\alpha\}\Big|\leq C|\mathbf{n}|^2 2^{-|\mathbf{n}|\frac{d-3-\beta}{3}}\alpha^{-\frac{2}{3}}.
	\end{equation*}
	
	We apply the Calder\'{o}n-Zygmund decomposition to $f_1$ at height $2^{|\mathbf{n}|\frac{d-3-\beta}{3}}\alpha^{\frac{2}{3}}$ to obtain the decomposition $f_1=g_1+h_1$ with a collection $\{Q_{\gamma}\}$ of disjoint cubes such that  
	\[\Vert g_1\Vert_{L^{\infty}}\leq 2^{|\mathbf{n}|\frac{d-3-\beta}{3}}\alpha^{\frac{2}{3}}, ~~h_1=\sum_{\gamma}h_{\gamma},~~\operatorname{supp}(h_{\gamma})\subset Q_{\gamma},\]
	\[\Vert h_{\gamma}\Vert_{L^{1}}\lesssim 2^{|\mathbf{n}|\frac{d-3-\beta}{3}}\alpha^{\frac{2}{3}}|Q_{\gamma}|~~ \text{ with}~~ \int_{Q_{\gamma}}h_{\gamma}=0, ~~\text{and}\]
	\[\sum_{\gamma}\big|Q_{\gamma}\big|\leq 2^{-|\mathbf{n}|\frac{d-3-\beta}{3}}\alpha^{-\frac{2}{3}}.\]
	Note that we have
	\begin{align*}
		\Big|\{x:\mathcal{M}^{\mathbf{n}}(f_1,f_2)(x)>\alpha\}\Big|\leq&\Big|\{x:\mathcal{M}^{\mathbf{n}}(h_1,f_2)(x)>\frac{\alpha}{2}\}\Big|+\Big|\{x:\mathcal{M}^{\mathbf{n}}(g_1,f_2)(x)>\frac{\alpha}{2}\}\Big|.
	\end{align*}
	\subsubsection*{\bf Contribution from $\mathcal{M}^{\mathbf{n}}(g_1,f_2)$: }
	This is the easy part. We use Chebyshev's inequality and \Cref{L2} to see that 
	\begin{align*}
		|\{x\in \mathbb{R}^{d}:\mathcal{M}^{\mathbf{n}}(g_1,f_2)(x)>\frac{\alpha}{2}\}|\lesssim \frac{2^{-|\mathbf{n}|\frac{2d-2-\beta}{2}}\|g_1\|_{L^2}\|f_2\|_{L^2}}{\alpha}.
	\end{align*}
	Since $\|g_1\|_2\lesssim2^{|\mathbf{n}|\frac{d-3-\beta}{6}}\alpha^{\frac{1}{3}}$, we obtain
	\begin{align*}
		|\{x\in \mathbb{R}^{d}:\mathcal{M}^{\mathbf{n}}(g_1,f_2)(x)>\frac{\alpha}{2}\}|&\lesssim 2^{-|\mathbf{n}|\frac{5d-3-2\beta}{6}}\alpha^{-\frac{2}{3}}\\
		&\lesssim2^{-|\mathbf{n}|\frac{d-3-\beta}{3}}\alpha^{-\frac{2}{3}}.
	\end{align*}
	\subsubsection*{\bf Contribution from $\mathcal{M}^{\mathbf{n}}(h_1,f_2)$:}
	Let $\tilde{Q}_{\gamma}$ denote the cube which is concentric with ${Q}_{\gamma}$ and has side length $l(\tilde{Q}_{\gamma})=10l(Q_{\gamma})$. We define the exceptional set 
	\[\mathcal{F}:=\Big(\cup_{\gamma}\tilde{Q}_{\gamma}\Big).\]
	Then, $|\mathcal{F}|\lesssim2^{-|\mathbf{n}|\frac{d-3-\beta}{3}}\alpha^{-\frac{2}{3}}$. Thus, by  Chebyshev's inequality, it is enough to show that
	\begin{equation}\label{other}
		\int_{\mathbb{R}^d\setminus \mathcal{F}}|\mathcal{M}^{\mathbf{n}}(h_1,f_2)(x)|^{\frac{2}{3}}~dx\lesssim |\mathbf{n}|^{2} 2^{-|\mathbf{n}|\frac{d-3-\beta}{3}}.
	\end{equation}
	We organize the function $h_1$ as follows
	\[h_{1}=\sum_{i\in\mathbb{Z}}h^{i}_{1},~~\text{where}~~h^{i}_{1}=\sum_{l(Q_{\gamma})=2^{-i}}h_{\gamma}.\]
	\begin{lemma}\label{badfunction2}
		Let $k,i\in\mathbb{Z}$ and $\mathbf{n}\in\mathbb{N}^2$. Then the following estimate holds uniformly in $k,i$ and $\mathbf{n}$.
		\begin{align*}
			\int\limits_{\mathbb{R}^{d}\setminus \mathcal{F}}\Big|\mathcal{M}^{\mathbf{n}}_k(h^{i}_1,f_2)(x)\Big|^{\frac{2}{3}}dx\lesssim n_2\,2^{-|\mathbf{n}|\frac{d-3-\beta}{3}}\min\left\{1,2^{\frac{2(i-k)}{3}},2^{\frac{2(k+n_1-i)}{3}}\right\}\Vert h^{i}_{1}\Vert^{\frac{2}{3}}_{L^{1}}\Vert f_{2}\Vert^{\frac{2}{3}}_{L^{2}}.
		\end{align*} 
	\end{lemma}
	We assume \Cref{badfunction2} for a moment and complete the proof. In fact, we can obtain the following estimate in \Cref{badfunction2}.
	\begin{align*}
		\int\limits_{\mathbb{R}^{d}\setminus \mathcal{F}}\Big|\mathcal{M}^{\mathbf{n}}_k(h^{i}_1,f_2)(x)\Big|^{\frac{2}{3}}dx
		&\lesssim n_2\,2^{-|\mathbf{n}|\frac{d-3-\beta}{3}}\min\left\{1,2^{\frac{2(i-k)}{3}},2^{\frac{2(k+n_1-i)}{3}}\right\}\Vert h^{i}_{1}\Vert^{\frac{2}{3}}_{L^{1}}\Vert \mathbf{R}_{k+n_2}f_{2}\Vert^{\frac{2}{3}}_{L^{2}}.
	\end{align*}
	Observe that
	\[\sup_{k\in\Z}\sum_{i\in\Z}\min\big\{1,2^{\frac{2(i-k)}{3}},2^{\frac{2(k+n_1-i)}{3}}\big\}=\sup_{i\in\Z}\sum_{k\in\Z}\min\big\{1,2^{\frac{2(i-k)}{3}},2^{\frac{2(k+n_1-i)}{3}}\big\}\lesssim n_1.\]
	
	Using the embedding $\ell^1\hookrightarrow\ell^\infty$ and the above estimates, we have
	\begin{align*}
		&\int_{\mathbb{R}^d\setminus \mathcal{F}}|\mathcal{M}^{\mathbf{n}}(h_1,f_2)(x)|^{\frac{2}{3}}dx\\
		&\leq \int_{\mathbb{R}^d\setminus \mathcal{F}}\left|\sum_{k\in\Z}\mathcal{M}^{\mathbf{n}}_k\Big(\sum_{i\in\Z}h_1^{i},f_2\Big)(x)\right|^{\frac{2}{3}}dx\\
		&\leq \sum_{k,i\in\Z}\int_{\mathbb{R}^d\setminus \mathcal{F}}\left|\mathcal{M}^{\mathbf{n}}_k\big(h_1^{i},f_2\big)(x)\right|^{\frac{2}{3}}dx\\
		&\lesssim \sum_{k,i\in\Z}n_2\,2^{-|\mathbf{n}|\frac{d-3-\beta}{3}}\min\Big\{1,2^{\frac{2(i-k)}{3}},2^{\frac{2(k+n_1-i)}{3}}\Big\}\Vert h^{i}_{1}\Vert^{\frac{2}{3}}_{L^{1}}\Vert \mathbf{R}_{k+n_2}f_{2}\Vert^{\frac{2}{3}}_{L^{2}}\\
		&\lesssim n_2\,2^{-|\mathbf{n}|\frac{d-3-\beta}{3}}\sum_{k\in\Z}\bigg(\sum_{i\in\Z}\min\Big\{1,2^{\frac{2(i-k)}{3}},2^{\frac{2(k+n_1-i)}{3}}\Big\}\bigg)^{\frac{1}{3}}\\
		&\hspace{3.2 cm}\times\bigg(\sum_{i\in\Z}\min\Big\{1,2^{\frac{2(i-k)}{3}},2^{\frac{2(k+n_1-i)}{3}}\Big\}\Vert h^{i}_{1}\Vert_{L^{1}}\bigg)^{\frac{2}{3}}\Vert \mathbf{R}_{k+n_2}f_{2}\Vert^{\frac{2}{3}}_{L^{2}}\\
		&\leq n_2\,n_1^{\frac{1}{3}}2^{-|\mathbf{n}|\frac{d-3-\beta}{3}}\bigg(\sum_{k\in\Z}\sum_{i\in\Z}\min\Big\{1,2^{\frac{2(i-k)}{3}},2^{\frac{2(k+n_1-i)}{3}}\Big\}\Vert h^{i}_{1}\Vert_{L^{1}}\bigg)^{\frac{2}{3}}\\
		&\hspace{4cm}\times\bigg(\sum_{k\in\Z}\Vert \mathbf{R}_{k+n_2}f_{2}\Vert^2_{L^{2}}\bigg)^{\frac{1}{3}}\\
		&\lesssim n_2\,n_1\,2^{-|\mathbf{n}|\frac{d-3-\beta}{3}}\bigg(\sum_{i\in\Z}\Vert h^{i}_{1}\Vert_{L^{1}}\bigg)^{\frac{2}{3}}\Vert f_{2}\Vert_{L^{2}}^{\frac{2}{3}}\\
		&\lesssim |\mathbf{n}|^2\,2^{-|\mathbf{n}|\frac{d-3-\beta}{3}},
	\end{align*}
	where we have used H\"older's inequality in the fourth and fifth steps with respect to $i$ and $k$ respectively.
	
	This proves \eqref{other} and completes the proof of part $(1)$ of \Cref{L2/3} under the assumption that \Cref{badfunction2} holds.
	\subsection*{Proof of \Cref{badfunction2}.}
	Using a scaling argument along with part $(1)$ of \Cref{l2/3}, we get
	\begin{align}\label{Mnk}
		\int_{\mathbb{R}^{d}\setminus \mathcal{F}}\Big|\mathcal{M}^{\mathbf{n}}_k(h^{i}_1,f_2)(x)\Big|^{\frac{2}{3}}dx&\lesssim2^{-|\mathbf{n}|\frac{d-3-\beta}{3}}\Vert \mathbf{R}_{k+n_1}h^{i}_{1}\Vert^{\frac{2}{3}}_{L^{1}}\Vert \mathbf{R}_{k+n_2}f_{2}\Vert^{\frac{2}{3}}_{L^{2}}\\
		&\nonumber\lesssim 2^{-|\mathbf{n}|\frac{d-3-\beta}{3}}\Vert h^{i}_{1}\Vert^{\frac{2}{3}}_{L^{1}}\Vert f_{2}\Vert^{\frac{2}{3}}_{L^{2}}.
	\end{align}
	Observe that this gives us the desired estimate for the terms satisfying $k\leq i<k+n_1$. 
	
	Next, we consider terms with $i\geq k+n_1$. We use the cancellation condition of  $h_\gamma$ to write
	\begin{align*}
		\mathbf{R}_{k+n_1}h_{\gamma}(x)&=\int_{Q_{\gamma}} \Big(\boldsymbol{\psi}_{2^{-(k+n_1)}}(x-y)-\boldsymbol{\psi}_{2^{-(k+n_1)}}(x-c_{Q})\Big)h_{\gamma}(y)dy\\
		&=-\int_{\mathbb{R}^{d}}\int^{1}_{0}\langle \nabla\boldsymbol{\psi}_{2^{-(k+n_1)}}(x-c_{Q}-t(y-c_{Q})),y-c_{Q}\rangle dt~h_{\gamma}(y) dy, 
	\end{align*}
	where $c_Q$ is the center of $Q_{\gamma}$. We observe that for $y\in Q_{\gamma}$, we have 
	\[\Big|\langle \nabla\boldsymbol{\psi}_{2^{-(k+n_1)}}(x-c_{Q}-t(y-c_{Q})),y-c_{Q}\rangle\Big|\lesssim {2^{(k+n_1-i)}}|(\nabla_{y}\boldsymbol{\psi})_{2^{-(k+n_1)}}(x-c_{Q}-t(y-c_{Q}))|.\]
	Applying Minkowski's integral inequality, we get 
	\begin{align*}
		\Vert \mathbf{R}_{k+n_1}h^{i}_{1}\Vert_{L^{1}}\lesssim 2^{k+n_1-i}\Vert h^{i}_{1}\Vert_{L^{1}}.
	\end{align*}
	Substituting the above estimate in \eqref{Mnk}, we get the desired bounds for $i\geq k+n_1$.
	
	Finally, for the remaining case when $i<k$, we exploit the restriction of the domain of integration to the complement of $\mathcal{F}$. Recall that $h^{i}_{1}=\sum\limits_{l(Q_{\gamma})=2^{-i}}h_{\gamma}$. Then, for $y'\in \supp(h_{\gamma})$ and $x\notin \mathcal{F}$, we have  
	\[\operatorname{dist}(x-ty, y')\geq c\,2^{-i}\quad \text{for all }t\in E_k,~ |y|\leq1,\]
	as $2^{-k}<2^{-i}$. Consequently, we have
	\begin{align*}
		&\|\mathcal{M}^{\mathbf{n}}_k(h^{i}_1,f_2)\|_{L^{2/3}(\mathbb{R}^{d}\setminus\mathcal{F})}^\frac{2}{3}\\
		&=\int_{\mathbb{R}^{d}\setminus\mathcal{F}}\sup_{t\in E_k}\Big|\int_{\mathbb{S}^{2d-1}}h_1^i\ast\psi_{2^{-k-n_1}}(x-ty)f_2\ast\psi_{2^{-k-n_2}}(x-tz)~d\sigma(y,z)\Big|^{\frac{2}{3}}dx\\
		&=\int_{\mathbb{R}^{d}\setminus\mathcal{F}}\sup_{t\in E_k}\bigg|\int_{\mathbb{S}^{2d-1}}\Big(\int_{\R^d}h_1^i(y')\psi_{2^{-k-n_1}}(x-ty-y')~dy'\Big)f_2\ast\psi_{2^{-k-n_2}}(x-tz)~d\sigma(y,z)\bigg|^{\frac{2}{3}}dx\\
		&=\int_{\mathbb{R}^{d}\setminus\mathcal{F}}\sup_{t\in E_k}\bigg|\int_{\mathbb{S}^{2d-1}}\Big(\int_{\R^d}h_1^i(y')\big(\psi_{2^{-k-n_1}}\mathlarger{\chi}_{\{|\cdot|\geq 2^{-i}\}}\big)(x-ty-y')~dy'\Big)f_2\ast\psi_{2^{-k-n_2}}(x-tz)~d\sigma(y,z)\bigg|^{\frac{2}{3}}dx\\
		&\leq\bigg\|\sup_{t\in E_k}\Big|\mathcal{A}_t(h^{i}_1\ast\varphi_{2^{-k-n_1}}^i,R_{k+n_2}f_2)\Big|\bigg\|_{L^{2/3}(\mathbb{R}^{d})}^\frac{2}{3},
	\end{align*}
	where $\varphi_{2^{-k-n_1}}^i(w)=\psi_{2^{-k-n_1}}(w)\mathlarger{\chi}_{\{|w|\geq c\,2^{-i}\}}(w)$.
	
	Using partition of unity from \Cref{sec:reductionlp}, we can decompose
	\[\sup_{t\in E_k}\Big|\mathcal{A}_t(h^{i}_1\ast\varphi_{2^{-k-n_1}}^i,R_{k+n_2}f_2)(x)\Big|\leq\sum_{m\geq0}\mathcal{M}^{(m,n_2)}_{k}(h^{i}_1\ast\varphi_{2^{-k-n_1}}^i,f_2)(x),\]
	where
	\[\mathcal{M}^{(m,n_2)}_{k}(h^{i}_1\ast\varphi_{2^{-k-n_1}}^i,f_2)(x)=\sup_{t\in E_k}\Big|\mathcal{A}_t\big(R_{k+m}(h^{i}_1\ast\varphi_{2^{-k-n_1}}^i),R_{k+n_2}f_2\big)(x)\Big|\]
	for $m\geq1$ and
	\[\mathcal{M}^{(0,n_2)}_{k}(h^{i}_1\ast\varphi_{2^{-k-n_1}}^i,f_2)(x)=\sup_{t\in E_k}\Big|\mathcal{A}_t\big(P_k(h^{i}_1\ast\varphi_{2^{-k-n_1}}^i),R_{k+n_2}f_2\big)(x)\Big|.\]
	For $m\geq0$, we claim that
	\begin{equation}\label{LPbad}
		\|\mathcal{M}^{(m,n_2)}_{k}(h^{i}_1\ast\varphi_{2^{-k-n_1}}^i,f_2)\|_{L^{2/3}}\lesssim2^{-|(m,n_2)|\frac{d-3-\beta}{2}}\|h^{i}_1\ast\varphi_{2^{-k-n_1}}^i\|_{L^1}\|f_2\|_{L^2},
	\end{equation}
	where $|(m,n_2)|=\sqrt{m^2+n_2^2}$.	Let us assume the estimate \eqref{LPbad} for a moment. Since $\psi\in\mathcal S(\R^d)$, we have $|\psi(w)|\lesssim_N(1+|w|)^{-N}$ for every $N>1$. Consequently, for $N\geq d+1$, we get that
	\begin{align*}
		\Vert \varphi_{2^{-k-n_1}}^i\Vert_{L^1}&=\int_{|w|\geq 2^{-i}}|\psi_{2^{-k-n_1}}(w)|\;dw\\
		&\lesssim\int_{|w|\geq 2^{-i}}2^{(k+n_1)d}(2+2^{k+n_1}|w|)^{-N}\;dw\\
		&=\int_{|w|\geq 2^{-i+k+n_1}}(2+|w|)^{-N}\;dw\\
		&\lesssim 2^{(N-d)(i-n_1-k)}\\
		&\leq 2^{-(N-d)n_1}2^{i-k},
	\end{align*}
	where the last inequality follows from the condition $i<k$.	We choose $N=2d$ to obtain
	\begin{align*}
		\|\mathcal{M}^{\mathbf{n}}_k(h^{i}_1,f_2)\|_{L^{2/3}(\mathbb{R}^{d}\setminus\mathcal{F})}^\frac{2}{3}
		&\leq\sum_{m\geq0}\|\mathcal{M}^{(m,n_2)}_{k}(h^{i}_1\ast\varphi_{2^{-k-n_1}}^i,f_2)\|_{L^{2/3}(\R^d)}^\frac{2}{3}\\
		&\lesssim\sum_{m\geq0}2^{-|(m,n_2)|\frac{d-3-\beta}{3}}\|h^{i}_1\ast\varphi_{2^{-k-n_1}}^i\|_{L^1}^\frac{2}{3}\|f_2\|_{L^2}^\frac{2}{3}\\
		&\lesssim\sum_{m\geq0}2^{-|(m,n_2)|\frac{d-3-\beta}{3}}2^{-\frac{2n_1d}{3}}2^{\frac{2(i-k)}{3}}\|h^{i}_1\|_{L^1}^\frac{2}{3}\|f_2\|_{L^2}^\frac{2}{3}\\
		&\lesssim 2^{-\frac{2n_1d}{3}}2^{\frac{2(i-k)}{3}}\bigg(\sum_{m=0}^{n_2-1}2^{-n_2\frac{d-3-\beta}{3}}\|h^{i}_1\|_{L^1}^\frac{2}{3}\|f_2\|_{L^2}^\frac{2}{3}+\sum_{m\geq n_2}2^{-m\frac{d-3-\beta}{3}}\|h^{i}_1\|_{L^1}^\frac{2}{3}\|f_2\|_{L^2}^\frac{2}{3}\bigg)\\
		&\leq n_22^{-\frac{2n_1d}{3}}2^{\frac{2(i-k)}{3}}2^{-n_2\frac{d-3-\beta}{3}}\|h^{i}_1\|_{L^1}^\frac{2}{3}\|f_2\|_{L^2}^\frac{2}{3}\\
		&\leq n_22^{-|\mathbf{n}|\frac{d-3-\beta}{3}}2^{\frac{2(i-k)}{3}}\|h^{i}_1\|_{L^1}^\frac{2}{3}\|f_2\|_{L^2}^\frac{2}{3}.
	\end{align*}
	This completes the proof of \Cref{badfunction2} given that the estimate \eqref{LPbad} holds.
	
	\subsection*{Proof of estimate \eqref{LPbad}:} Due to the scaling argument in \eqref{scaling}, it is enough to prove \eqref{LPbad} for $\mathcal{M}^{(m,n_2)}_{0}$. Hence, when $m\geq1$, the desired estimate in \eqref{LPbad} follows from \Cref{l2/3}. Thus, it remains to prove \eqref{LPbad} for $m=0$.
	
	Let $b_1=h^{i}_1\ast\varphi_{2^{-k-n_1}}^i$ and denote $\mathcal{M}^{(0,n_2)}_{0}=\mathcal{M}^{n_2}_0$. Performing a decomposition similar to that in \eqref{localdecomposition}, we have
	\[\mathcal{M}^{n_2}_0(b_1,f_2)(x)\leq \mathbb{I}^{n_2}(x)+\mathbb{II}^{n_2}(x)+\mathbb{III}^{n_2}(x),\]
	where
	\[\mathbb{I}^{n_2}(x)=\sup_{t\in E_0}\left|\int_{\R^d\times\R^d}\frac{a_0 e^{it|(\xi,\eta)|}}{(t|(\xi,\eta)|)^{\frac{2d-1}{2}}}\widehat{\phi}\left( \xi\right)\widehat{\psi}\left(2^{-n_2} \eta\right)\widehat{b_1}(\xi)\widehat{f_2}(\eta)e^{ix\cdot(\xi+\eta)}~d\xi d\eta\right|,\]
	\[\mathbb{II}^{n_2}(x)=\sup_{t\in E_0}\left|\int_{\R^d\times\R^d}\frac{b_0 e^{-it|(\xi,\eta)|}}{(t|(\xi,\eta)|)^{\frac{2d-1}{2}}}\widehat{\phi}\left( \xi\right)\widehat{\psi}\left(2^{-n_2} \eta\right)\widehat{b_1}(\xi)\widehat{f_2}(\eta)e^{ix\cdot(\xi+\eta)}~d\xi d\eta\right|,\]
	and
	\[\mathbb{III}^{n_2}(x)=\sup_{t\in E_0}\left|\int_{\R^d\times\R^d}\mathfrak{m}_{t}^{n_2}(\xi,\eta)\widehat{b_1}(\xi)\widehat{f_2}(\eta)e^{ix\cdot(\xi+\eta)}~d\xi d\eta\right|.\]
	Here $\mathfrak{m}_{t}^{n_2}(\xi,\eta)=\Big[\widehat{d\sigma}(t\xi,t\eta)-\Big(\frac{a_0 e^{it|(\xi,\eta)|}}{(t|(\xi,\eta)|)^{\frac{2d-1}{2}}}+\frac{b_0 e^{-it|(\xi,\eta)|}}{(t|(\xi,\eta)|)^{\frac{2d-1}{2}}}\Big)\Big]\widehat{\phi}\left(\xi\right)\widehat{\psi}\left(2^{-n_2} \eta\right)$. Using the estimate \eqref{asymptotics} with $M=0$, and the support condition of $\widehat{\phi}$ and $\widehat{\psi}$, we have
	\[\big|\partial^{\alpha}_{\xi}\partial^{\gamma}_{\eta}\mathfrak{m}_{t}^{n_2}(\xi,\eta)\big|\lesssim2^{-n_2\frac{2d-1}{2}}|\eta|^{-|\gamma|}, \text{ for all }t\in E_0.\]
	This yields that 
	\begin{equation*}\label{error2}
		\mathbb{III}^{n_2}(x)\lesssim 2^{-n_2\frac{2d-1}{2}}M_{HL}b_1(x)M_{HL}f_2(x).
	\end{equation*}
	Therefore, $\mathbb{III}^{n_2}$ satisfies the desired bounds in \eqref{LPbad}. 
	
	Since the arguments used to address the terms $\mathbb{I}^{n_2}$ and $\mathbb{II}^{n_2}$ are similar, we provide the proof of \eqref{LPbad} for the term $\mathbb{I}^{n_2}$.
	
	Let $\mathcal{E}_0$ be the disjoint collection of intervals of length $2^{-n_2}$ covering $E_0$. Decomposing $\mathbb{I}^{n_2}$ similar to that of $\mathbb{I}^{\mathbf{n}}$, we have
	\begin{align*}
		\mathbb{I}^{n_2}(x)&\lesssim\sum_{N=2^{n_2-1}}^{2^{n_2+1}}\sup_{I\in\mathcal{E}_0}\sup_{t\in I}\Big|T^{n_2}_{t,N}\Big(b_1,f_2\Big)(x) \Big|,
	\end{align*}
	where
	\[T^{n_2}_{t,N}\Big(b_1,f_2\Big)(x):=\int_{\R^d\times\R^d}\frac{e^{it|(\xi,\eta)|}}{(t|(\xi,\eta)|)^{\frac{2d-1}{2}}}\widehat{\phi}\left( \xi\right)\widehat{\psi}\left(2^{-n_2} \eta\right)\widehat{b_1}(\xi)\widehat{F_Nf_2}(\eta)e^{ix\cdot(\xi+\eta)}~d\xi d\eta,\]    
	and $F_N$ is defined as in \eqref{operatorF}.
	
	Then, a similar argument to that of $\mathbb{I}^{\mathbf{n}}$ in \Cref{Mn0} gives us
	\begin{align*}\label{localdeco2}
		\mathbb{I}^{n_2}(x)\lesssim&\sum_{N=2^{n_2-1}}^{2^{n_2+1}}2^{\frac{n_2}{2}}\mathcal{G}^{n_2}_{0,N}(b_1,f_2)(x)+\sum_{N=2^{n_2-1}}^{2^{n_2+1}}\big(\mathcal{G}^{n_2}_{0,N}(b_1,f_2)(x)\tilde{\mathcal{G}}^{n_2}_{0,N}(b_1,f_2)(x)\big)^{\frac{1}{2}},
	\end{align*}
	where
	\[\mathcal{G}^{n_2}_{0,N}(b_1,f_2)(x):=\left(\sum_{I\in \mathcal{E}_0}\int_I|T^{n_2}_{s,N}(b_1,f_2)(x)|^2\;ds\right)^\frac{1}{2}~\text{ and}\]
	\[\tilde{\mathcal{G}}^{n_2}_{0,N}(b_1,f_2)(x):=\left(\sum_{I\in \mathcal{E}_0}\int_I\left|\frac{\partial}{\partial s}T^{n_2}_{s,N}(b_1,f_2)(x)\right|^2\;ds\right)^\frac{1}{2}.\]
	
	Fix a positive integer $N$ such that $2^{n_2-1}\leq N\leq2^{n_2+1}$ and write
	\begin{align*}
		e^{it|(\xi,\eta)|}=e^{it\tilde{\Phi}_N(\xi,\eta)}e^{it\sqrt{|\xi|^2+N^2}},
	\end{align*}
	where
	\[\tilde{\Phi}_N(\xi,\eta):=|(\xi,\eta)|-\sqrt{|\xi|^2+N^2}=\frac{(|\eta|-N)(|\eta|+N)}{\sqrt{|\xi|^2+|\eta|^2}+\sqrt{|\xi|^2+N^2}}.\]
	Note that due to support of $\widehat{\phi}$ and $\widehat{\psi}$, we have $|\xi|\leq2$ and $2^{n_2-1}\leq|\eta|\leq2^{n_2+1}, n_2\geq1$. Thus, $\tilde{\Phi}_N$ satisfies
	\begin{equation}\label{deri2}
		|\partial^\alpha_\xi\tilde{\Phi}_N(\xi,\eta)|\lesssim1,
	\end{equation}
	for all multi-indices $\alpha$. Then, we can write
	\begin{align*} T^{n_2}_{t,N}(b_1,f_2)(x)&=\int_{\R^d\times\R^d}m_t^{n_2}(\xi,\eta)\widehat{U_{t,N}b_1}(\xi)\widehat{F_Nf_2}(\eta)e^{ix\cdot(\xi+\eta)}\;d\xi d\eta,
	\end{align*}
	where
	\begin{equation*}\label{operatorU2}
		\widehat{U_{t,N}b_1}(\xi):=\widehat{b_1}(\xi) \widehat{\phi}\left(\xi\right) e^{it\left(\sqrt{|\xi|^2+N^{2}}\right)},
	\end{equation*}
	and
	\[m^{n_2}_{t}(\xi, \eta) :=\frac{e^{it\tilde{\Phi}_N(\xi,\eta)}}{|t(\xi,\eta)|^{\frac{2d-1}{2}}}\widehat{\phi}\left(\xi\right) \widehat{\psi}\left(2^{-n_2} \eta\right).\]
	Using the estimate \eqref{deri2}, we get that
	\begin{equation}\label{derivative2}
		|\partial^\alpha_\xi \big(m^{n_2}_{t}(\xi, \eta)\big)|\lesssim C_{\alpha}2^{-\frac{2d-1}{2}n_2}, 
	\end{equation}    
	for all multi-indices $\alpha$ uniformly in $t\in E_0$. 
	
	Since $\widehat{U_{t,N}b_1}(\xi)=\int_{\R^d}U_{t,N}b_1(y)e^{-i\xi\cdot y}~dy$, we can rewrite
	\begin{align}\label{factor2}
		T^{n_2}_{t,N}(b_1,f_2)(x)&=\int_{\R^d\times\R^d}\widehat{F_Nf_2}(\eta)U_{t,N}b_1(y)\Big(\int_{\R^d}m^{n_2}_{t}(\xi, \eta)e^{i\xi\cdot(x-y)}d\xi\Big)e^{ix\cdot\eta}\;d\eta dy.
	\end{align}
	Similarly to \eqref{ip}, an integration by parts argument gives
	\begin{align}\label{ip2}
		&\int_{\R^d}m^{n_2}_{t}(\xi, \eta)e^{i\xi\cdot(x-y)}d\xi=\frac{1}{(1+|x-y|^2)^{\lceil\frac{d}{2}\rceil+1}}\int_{\R^d}(1-\Delta_\xi)^{\lceil\frac{d}{2}\rceil+1}\Big(m^{n_2}_{t}(\xi, \eta)\Big)e^{i\xi\cdot(x-y)}d\xi.
	\end{align}
	Substituting \eqref{ip2} into \eqref{factor2}, we get
	\begin{align*}
		T^{n_2}_{t,N}(b_1,f_2)(x)&=\int_{\mathbb{R}^{d}}A\left(U_{t,N}b_1\right)(x,\xi) B^{n_2}_{t, N}\left(F_{N}f_2\right)(x, \xi) d \xi,
	\end{align*}
	where
	\begin{align*}\label{operatorB2}
		B^{n_2}_{t, N}\left(F_{N}f_2\right)(x, \xi):=&\int_{\mathbb{R}^{d}} \left(1- \Delta_{\xi}\right)^{\lceil \frac{d}{2}\rceil+1}\left(m^{n_2}_{t}(\xi, \eta)\right)\widehat{F_{N}f_2}(\eta)e^{ix \cdot \eta} d\eta,
	\end{align*}
	and
	\[A\left(U_{t, N}b_1\right)(x,\xi):=\int_{\mathbb{R}^{d}} \frac{U_{t, N}b_1(y)e^{i\xi\cdot(x-y)}}{\left(1+\left|x-y\right|^{2}\right)^{\lceil \frac{d}{2}\rceil+1}} d y.\]
	Note that
	\begin{align*}
		|A\left(U_{t, N}b_1\right)(x,\xi)|&\leq\int_{\mathbb{R}^{d}} \frac{|U_{t, N}b_1(y)|}{\left(1+\left|x-y\right|^{2}\right)^{\lceil \frac{d}{2}\rceil+1}} d y\\
		&:=L\left(U_{t,N}b_1\right)(x).
	\end{align*}
	Then, we have
	\begin{equation*}\label{decomposition2}
		|T^{n_2}_{t,N}(b_1,f_2)(x)|\leq L\left(U_{t,N}b_1\right)(x) \int_{\mathbb{R}^{d}}\left|B^{n_2}_{t, N}\left(F_{N}f_2\right)(x, \xi)\right| d\xi,
	\end{equation*}
	and
	\begin{align*}
		\mathcal{G}^{n_2}_{0,N}(b_1,f_2)(x)\lesssim\left(\int_{\mathbb{R}^{d}} \sup_{1 \leq s\leq 2}\left|B^{n_2}_{s, N}\left(F_Nf_2\right)(x, \xi)\right| d \xi\right)\bigg(\sum_{I\in \mathcal{E}_0}\int_I\left|L\left(U_{s,N}b_1\right)(x)\right|^2\;ds\bigg)^\frac{1}{2}.
	\end{align*}
	Recall that $\supp(\widehat{\phi})\subset B(0,2)$. Similar to the $L^p-$estimates in \Cref{bn} along with \eqref{derivative2}, we can get that
	\[\int_{\mathbb{R}^{d}}\left\| \sup_{1 \leq s\leq 2}\left|B^{n_2}_{s, N}\left(F_{N}f_2\right)(\cdot, \xi)\right|\right\|_{L^2} d \xi\lesssim2^{-n_2\frac{2d-1}{2}}\|F_Nf_2\|_{L^2}.\]
	Using Minkowski's integral inequality twice, we have
	\begin{align*}
		&\bigg(\sum_{I\in \mathcal{E}_0}\int_I\left|L\left(U_{s,N}b_1\right)(x)\right|^2\;ds\bigg)^\frac{1}{2}\\&\leq2^{n_2\frac{\beta}{2}}\int_{\R^d}\bigg(\int_I\left|U_{s,N}b_1(x-y)\right|^2\;ds\bigg)^\frac{1}{2}\frac{1}{\left(1+\left|y\right|^{2}\right)^{\lceil \frac{d}{2}\rceil+1}}dy\\
		&\leq2^{n_2\frac{\beta}{2}}\int_{\R^d}\int_{\R^d}|b_1(x-y-z)|\Big(\int_I\left|K_{s,N}(z)\right|^2\;ds\Big)^\frac{1}{2}\frac{1}{\left(1+\left|y\right|^{2}\right)^{\lceil \frac{d}{2}\rceil+1}}~dzdy,
	\end{align*}
	where $K_{s,N}(z)=\int_{\R^d}\widehat{\phi}\left(\xi\right) e^{is\left(\sqrt{|\xi|^2+N^{2}}\right)}e^{iz\cdot\xi}d\xi$. Since $\phi\in\mathcal{S}(\R^d)$ and $s\in I\subset[1,2]$, we have good kernel estimates for $K_{s,N}$ which are uniform in $s$. Thus, we have
	\begin{align*}
		\left\|\bigg(\sum_{I\in \mathcal{E}_0}\int_I\left|L\left(U_{s,N}b_1\right)\right|^2\;ds\bigg)^\frac{1}{2}\right\|_{L^1}\lesssim2^{-n_2\frac{1-\beta}{2}}\|b_1\|_{L^1}.
	\end{align*}
	Using H\"older's inequality and the estimates above, we get that
	\begin{align*}
		\Bigg\|\sum_{N=2^{n_2-1}}^{2^{n_2+1}}\mathcal{G}^{n_2}_{0,N}(b_1,f_2)\Bigg\|_{L^{2/3}}&\lesssim2^{-n_2\frac{2d-\beta}{2}}\bigg(\sum_{N=2^{n_2-1}}^{2^{n_2+1}}\|b_1\|_{L^1}^{\frac{2}{3}}\|F_Nf_2\|_{L^2}^{\frac{2}{3}}\bigg)^{\frac{3}{2}}\\
		&\lesssim2^{-n_2\frac{2d-\beta-2}{2}}\|b_1\|_{L^1}\|f_2\|_{L^2}.
	\end{align*}
	Similarly, we can get
	\begin{align*}
		\Bigg\|\sum_{N=2^{n_2-1}}^{2^{n_2+1}}\tilde{\mathcal{G}}^{n_2}_{0,N}(b_1,f_2)\Bigg\|_{L^{2/3}}&\lesssim2^{-n_2\frac{2d-\beta-4}{2}}\|b_1\|_{L^1}\|f_2\|_{L^2}.
	\end{align*}
	Consequently, we have
	\[\big\|\mathcal{M}^{(0,n_2)}_0(b_1,f_2)\big\|_{L^{2/3}}\lesssim2^{-n_2\frac{2d-\beta-3}{2}}\|b_1\|_{L^1}\|f_2\|_{L^2}.\]
	This proves the desired bounds in \eqref{LPbad} for $\mathcal{M}^{(0,n_2)}_0$. 	\qed
	
	\qed
	
	\subsection{Proof of \Cref{L1}.}
	Let us assume without loss of generality that $\|f_1\|_{L^1}=\|f_2\|_{L^1}=1$. Hence, for given  $\alpha>0$, we are required to prove that 
	\begin{equation}\label{weak}
		\Big|\{x:\mathcal{M}^{\mathbf{n}}(f_1,f_2)(x)>\alpha\}\Big|\leq C|\mathbf{n}|^{2}2^{|\mathbf{n}|\frac{\beta}{2}}\alpha^{-\frac{1}{2}}.
	\end{equation}
	We apply the Calder\'{o}n-Zygmund decomposition to both functions $f_1$ and $f_2$ with the same choice of height $\alpha^{\frac{1}{2}}$. This gives us the decomposition $f_j=g_j+h_j$ with collections $\{Q_{\gamma_j}\}$ of disjoint cubes for $ j=1,2$ such that 
	\[\Vert g_j\Vert_{L^{\infty}}\leq \alpha^{\frac{1}{2}}, ~~h_j=\sum_{\gamma_j}h_{\gamma_j},~~\supp(h_{\gamma_j})\subset Q_{\gamma_j},\]
	\begin{align*}
		&\Vert h_{\gamma_j}\Vert_{L^{1}}\leq C\alpha^{\frac{1}{2}}|Q_{\gamma_j}| ~~ \text{with}~~ \int_{Q_{\gamma_j}}h_{\gamma_j}=0, ~~\text{and} ~~ \sum_{\gamma_j}\big|Q_{\gamma_j}\big|\leq \alpha^{-\frac{1}{2}}.
	\end{align*}
	We have
	\begin{align*}
		\Big|\{x:\mathcal{M}^{\mathbf{n}}(f_1,f_2)(x)>\alpha\}\Big|\leq&\Big|\{x:\mathcal{M}^{\mathbf{n}}(h_1,h_2)(x)>\frac{\alpha}{4}\}\Big|+\Big|\{x:\mathcal{M}^{\mathbf{n}}(h_1,g_2)(x)>\frac{\alpha}{4}\}\Big|\\
		&+\Big|\{x:\mathcal{M}^{\mathbf{n}}(g_1,h_2)(x)>\frac{\alpha}{4}\}\Big|+\Big|\{x:\mathcal{M}^{\mathbf{n}}(g_1,g_2)(x)>\frac{\alpha}{4}\}\Big|.
	\end{align*}
	We need to prove the estimate in \eqref{weak} for each of the four terms in the equation above.
	\subsubsection*{\bf Contribution from $\mathcal{M}^{\mathbf{n}}(g_1,g_2)$: }
	From Chebyshev's inequality and \Cref{L2}, we get that
	\begin{align*}
		|\{x\in \mathbb{R}^{d}:\mathcal{M}^{\mathbf{n}}(g_1,g_2)(x)>\frac{\alpha}{4}\}|&\lesssim \frac{2^{-|\mathbf{n}|\frac{2d-2-\beta}{2}}\|g_1\|_{L^2}\|g_2\|_{L^2}}{\alpha}\\
		&\lesssim \alpha^{-\frac{1}{2}},
	\end{align*}
	where we used $\|g_1\|_{L^2}\lesssim\alpha^{\frac{1}{4}}$ and $\|g_2\|_{L^2}\lesssim\alpha^{\frac{1}{4}}$ in the last step. Thus, we obtain \eqref{weak} for $\mathcal{M}^{\mathbf{n}}(g_1,g_2)$.
	
	\subsubsection*{\bf Contributions from $\mathcal{M}^{\mathbf{n}}(h_1,g_2)$ and $\mathcal{M}^{\mathbf{n}}(g_1,h_2)$: }
	Due to symmetry, it is enough to provide the arguments for the term $\mathcal{M}^{\mathbf{n}}(h_1,g_2)$. From arguments as in the proof of \Cref{trivialestimates}, we can get that $\mathcal{M}^{\mathbf{n}}$ maps $L^1\times L^\infty$ into $L^{1,\infty}$. Thus, Chebyshev's inequality gives us
	\begin{align*}
		|\{x\in \mathbb{R}^{d}:\mathcal{M}^{\mathbf{n}}(h_1,g_2)(x)>\frac{\alpha}{4}\}|&\lesssim \frac{\|h_1\|_{L^1}\|g_2\|_{L^\infty}}{\alpha}\\
		&\lesssim \alpha^{-\frac{1}{2}},
	\end{align*}
	where we have used the fact $\|g_2\|_{L^\infty}\lesssim\alpha^{\frac{1}{2}}$ in the last inequality. This proves the desired estimates for $\mathcal{M}^{\mathbf{n}}(h_1,g_2)$.
	\subsubsection*{\bf Contribution from $\mathcal{M}^{\mathbf{n}}(h_1,h_2)$:}
	We use the notation as in the previous section. We define the exceptional set 
	\[\mathcal{F}:=\Big(\cup_{\gamma_1}\tilde{Q}_{\gamma_1}\Big)\cup \Big(\cup_{\gamma_2}\tilde{Q}_{\gamma_2}\Big).\]
	Then, $|\mathcal{F}|\lesssim\alpha^{-\frac{1}{2}}$. Thus, using Chebyshev's inequality, it is enough to show that 
	\[\int_{\mathbb{R}^d\setminus \mathcal{F}}|\mathcal{M}^{\mathbf{n}}(h_1,h_2)(x)|^{\frac{1}{2}}~dx\lesssim |\mathbf{n}|^{2}2^{|\mathbf{n}|\frac{\beta}{2}}.\]
	We organize the bad functions  
	\[h_{j}=\sum_{i_j\in\mathbb{Z}}h^{i_j}_{j},~~\text{where}~~h^{i_j}_{j}=\sum_{l(Q_{\gamma_j})=2^{-i_j}}h_{\gamma_j}.\]
	\begin{lemma}\label{badfunction}
		Let $k,i_1,i_2\in\mathbb{Z}$ and $\mathbf{n}\in\mathbb{N}^2$. Then the following estimate holds uniformly in $k,i_j,n_j$.
		\begin{align*}
			\int\limits_{\mathbb{R}^{d}\setminus \mathcal{F}}\Big|\mathcal{M}^{\mathbf{n}}_k(h^{i_1}_1,h^{i_2}_2)(x)\Big|^{\frac{1}{2}}dx\lesssim \min_{j=1,2}\{1,2^{\frac{i_j-k}{2}},2^{\frac{k+n_j-i_j}{2}}\}2^{|\mathbf{n}|\frac{\beta}{2}}\prod^{2}_{j=1}\Vert h^{i_j}_{j}\Vert^{\frac{1}{2}}_{L^{1}}.
		\end{align*} 
	\end{lemma}
	We use similar arguments as in the previous section along with \Cref{badfunction} and Lemma~5.2 from \cite{ChristZhou} to conclude that 
	\[\sum_{k\in\mathbb{Z}}\sum_{i_1,i_2\in\mathbb{Z}}\min_{j=1,2}\min\Big\{2^{\frac{i_j-k}{2}},2^{\frac{k+n_j-i_j}{2}},1\Big\}\prod^{2}_{j=1}\Vert h^{i_j}_{j}\Vert^{\frac{1}{2}}_{L^{1}}\lesssim |\mathbf{n}|^{2}.\]
	This completes the proof for the parts where we have bad functions in both places.  
	\subsection*{Proof of \Cref{badfunction}.}
	The proof is similar to that of \Cref{badfunction2} with some modifications. We will give a brief sketch only. Note that from \Cref{l1}, we have
	\begin{align*}
		\int_{\mathbb{R}^{d}\setminus \mathcal{F}}\Big|\mathcal{M}^{\mathbf{n}}_k(h^{i_1}_1,h^{i_2}_2)(x)\Big|^{\frac{1}{2}}dx&\lesssim 2^{|\mathbf{n}|\frac{\beta}{2}}\Vert h^{i_1}_{1}\Vert^{\frac{1}{2}}_{L^{1}}\Vert h^{i_2}_{2}\Vert^{\frac{1}{2}}_{L^{1}}.
	\end{align*}
	The estimate above gives us the desired bound when $k\leq i_j\leq k+n_j$. \\
	When $i_j>k+n_j$, we write $h^{i_j}_j=\sum\limits_{l(Q_{\gamma_j})=2^{-i_j}}h_{\gamma_j}$. Using the mean zero property of the functions $h_{\gamma_j}$, we get
	\begin{align}\label{min}
		\Vert \mathbf{R}_{k+n_j}h^{i_j}_{j}\Vert_{L^{1}}\lesssim 2^{k+n_j-i_j}\Vert h^{i_j}_{j}\Vert_{L^{1}}.
	\end{align}
	When $i_j<k$, we follow the idea as used in the proof of \Cref{l1} with some modifications.	Write $\R^d=\cup_{Q\in\mathfrak C}Q$, where $\mathfrak{C}$ is a family of disjoint cubes with side length $2^{-k}$ and having sides parallel to the coordinate axes. Then
	\begin{align*}
		&\|\mathcal{M}^{\mathbf{n}}_k(h^{i_1}_1,h^{i_2}_2)\|_{L^{1/2}(\mathbb{R}^{d}\setminus\mathcal{F})}^\frac{1}{2}\\	
		&=\int_{\mathbb{R}^{d}\setminus\mathcal{F}}\bigg(\sup\limits_{t\in E_k}\bigg|\int_{\mathbb{S}^{2d-1}}\bigg(\sum\limits_{Q\in\mathfrak{C}}(h^{i_1}_1*\psi_{2^{-k-n_1}})\mathlarger{\chi}_{_Q}\bigg)(x-ty)( h^{i_2}_2*\psi_{2^{-k-n_2}})(x-tz)\;d\sigma(y,z)\bigg|\bigg)^\frac{1}{2}dx\\
		&\leq\sum_{Q\in\mathfrak{C}}\int_{5Q\setminus\mathcal{F}}\bigg(\sup\limits_{t\in E_k}\int_{\mathbb{S}^{2d-1}}\big((|h^{i_1}_1|*\Psi_{2^{-k-n_1}})\mathlarger{\chi}_{_Q}\big)(x-ty)\big(( |h^{i_2}_2|*\Psi_{2^{-k-n_2}})\mathlarger{\chi}_{_{10Q}}\big)(x-tz)\;d\sigma(y,z)\bigg)^\frac{1}{2}dx.
	\end{align*}
	where $\Psi(w)=(2+|w|)^{-N}$ so that $|\psi(w)|\lesssim\Psi(w)$. 
	Recall that $h^{i_1}_{1}=\sum\limits_{l(Q_{\gamma_1})=2^{-i_1}}h_{\gamma_1}$. Then, for $y'\in \supp(h_{\gamma_1})$ and $x\notin \mathcal{F}$, we have  
	\[\operatorname{dist}(x-ty, y')\geq c\,2^{-i_1}\quad \text{for all }t\in E_k,~ |y|\leq1,\]
	as $2^{-k}<2^{-i_1}$. 
	
	We define $\Phi_{2^{-k-n_1}}(w)=\Psi_{2^{-k-n_1}}(w)\mathlarger{\chi}_{\{|w|\geq c\,2^{-i_1}\}}(w)$. Therefore, when $x\notin \mathcal{F}$, we have
	\[|h^{i_1}_1|*\Psi_{2^{-k-n_1}}(x-ty)\lesssim|h^{i_1}_1|*\Phi_{2^{-k-n_1}}(x-ty)\quad \text{for all }t\in E_k,~ |y|\leq1.\]
	Recall that for any $N>d$, we have
	\begin{align*}
		\Vert \Phi_{2^{-n_1-k}}\Vert_{L^1}&\lesssim 2^{(N-d)(i_1-n_1-k)}.
	\end{align*}
	Let $\mathcal{E}_k$ be the collection of disjoint intervals of length $2^{-k-|\mathbf{n}|}$ covering $E_k$. Recall that \[\Psi_{2^{-k-n_1}}(w_1)\lesssim\Psi_{2^{-k-n_1}}(w_2),~\text{ whenever}~ |w_1-w_2|\leq2^{-k-|\mathbf{n}|}.\]
	Now, using the Cauchy-Schwarz inequality, we have
	\begin{align*}
		&\|\mathcal{M}^{\mathbf{n}}_k(h^{i_1}_1,h^{i_2}_2)\|_{L^{1/2}(\mathbb{R}^{d}\setminus\mathcal{F})}^\frac{1}{2}\\
		&\lesssim\sum_{Q\in\mathfrak{C}}2^{-\frac{kd}{2}}\bigg(\int_{5Q\setminus\mathcal{F}}\sum_{I\in\mathcal{E}_k}\sup_{t\in I}\int_{\mathbb S^{2d-1}}\left((|h^{i_1}_1|*\Psi_{2^{-k-n_1}})\mathlarger{\chi}_{_{Q}}\right)(x-ty)\\
		&\hspace{4.4cm}\times\left(( |h^{i_2}_2|*\Psi_{2^{-k-n_2}})\mathlarger{\chi}_{_{10Q}}\right)(x-tz)\;d\sigma(y,z)\;dx\bigg)^\frac{1}{2}\\
		&\lesssim\sum_{Q\in\mathfrak{C}}2^{-\frac{kd}{2}}\bigg(\int_{5Q}\sum_{I\in\mathcal{E}_k}\int_{\mathbb S^{2d-1}}\left((|h^{i_1}_1|*\Phi_{2^{-k-n_1}})\mathlarger{\chi}_{_{2Q}}\right)(x-t_0y)\\
		&\hspace{4.4cm}\times\left(( |h^{i_2}_2|*\Psi_{2^{-k-n_2}})\mathlarger{\chi}_{_{6Q}}\right)(x-t_0z)\;d\sigma(y,z)\;dx\bigg)^\frac{1}{2}\\
		&\lesssim2^{-\frac{kd}{2}}\sum_{Q\in\mathfrak{C}}\bigg(\sum_{I\in\mathcal{E}_k}t_0^{-d}\big\|(|h^{i_1}_1|*\Phi_{2^{-k-n_1}})\mathlarger{\chi}_{_{2Q}}\big\|_{L^1}\;\big\|( |h^{i_2}_2|*\Psi_{2^{-k-n_2}})\mathlarger{\chi}_{_{6Q}}\big\|_{L^1}\bigg)^\frac{1}{2}\\
		&\leq2^{-\frac{kd}{2}}2^{|\mathbf{n}|\frac{\beta}{2}}2^{\frac{kd}{2}}\bigg(\sum_{Q\in\mathfrak{C}}\|(|h^{i_1}_1|*\Phi_{2^{-k-n_1}})\mathlarger{\chi}_{_{2Q}}\|_{L^1}\bigg)^\frac{1}{2}\bigg(\sum_{Q\in\mathfrak{C}}\|(|h^{i_2}_2|*\Psi_{2^{-k-n_2}})\mathlarger{\chi}_{_{6Q}}\|_{L^1}\bigg)^\frac{1}{2}\\
		&\lesssim2^{|\mathbf{n}|\frac{\beta}{2}}2^{(N-d)\frac{i_1-k-n_1}{2}}\|h^{i_1}_{1}\|_{L^1}^\frac{1}{2}\|h^{i_2}_2\|_{L^1}^\frac{1}{2},
	\end{align*}
	where we have used $L^1\times L^1\to L^1-$boundedness of single-scale bilinear spherical average in the third inequality and H\"older's inequality in the fourth inequality in the expression above. \\
	Since $i_1<k$, we have $2^{(N-d)\frac{i_1-k-n_1}{2}}\leq 2^{\frac{i_1-k}{2}}$ for every $N\geq d+1$. We choose $N=d+1$ to obtain
	\[\Vert \mathcal{M}^{\mathbf{n}}_k(h^{i_1}_1,h^{i_2}_2)\Vert_{L^{1/2}(\mathbb{R}^{d}\setminus\mathcal{F})}^\frac{1}{2}\lesssim 2^{\frac{i_1-k}{2}}2^{|\mathbf{n}|\frac{\beta}{2}}\|h^{i_1}_1\|_{L^1}^\frac{1}{2}\|h^{i_2}_2\|_{L^1}^\frac{1}{2}.\]
	Interchanging the roles of $i_1$ and $i_2$ proves the desired bound for $i_2$. This completes the proof of \Cref{badfunction}.  Consequently, the proof of \Cref{L1} is completed.
	\qed
	
	\section{Proof of \Cref{main2}}\label{d=1}
	First, observe that for any set $E\subseteq \R_+$, the maximal function $\mathcal{M}_E(f_1,f_2)$ is dominated in pointwise sense by the full maximal function $\mathcal{M}_{full}(f_1,f_2)$. Therefore, invoking the $L^p-$boundedness  results of the operator $\mathcal{M}_{full}$, see \cite[Theorem 1.3]{ChristZhou} and \cite[Theorem 1]{DosidisRamos}, we get that the maximal operator $\mathcal{M}_E$ is bounded from $L^{p_1}\times L^{p_2}\to L^{p}$ for all $p_1, p_2>2$. 
	
	We have 
	\begin{align*}
		\mathcal{M}_{E}(f_1,f_2)(x)\leq \Vert f_2\Vert_{L^\infty} \sup_{t\in E}\left|\int_{\mathbb{S}^{1}}f_1(x+ty)~d\sigma(y,z)\right|.
	\end{align*}
	Observe that in order to study $L^p$ estimates of the maximal operator 
	$$f_1\rightarrow \sup_{t\in E}\left|\int_{\mathbb{S}^{1}}f_1(x+ty)~d\sigma(y,z)\right|,$$ 
	it is enough to study the maximal operator
	\[\mathcal{N}f_1(x)=\sup_{t\in E}\int^{1}_{0}|f_1(x+tr)|~\frac{dr}{\sqrt{1-r^2}}.\]
	By H\"older's inequality $\frac{1}{q}+\frac{1}{q'}=1,~q\geq1$, we have
	\begin{align*}
		\mathcal{N}f_1(x)
		&\leq \sup_{t\in E}\sum_{j=1}^\infty 2^{j\left(\frac{1}{2}-\frac{1}{q'}\right)}\left(\int_{1-2^{-(j-1)}}^{1-2^{-j}}|f_1(x+tr)|^q~dr\right)^{\frac{1}{q}}\\
		&\leq\sum_{j=1}^\infty 2^{-\frac{j}{2}}\sup_{k\in \mathbb{Z}}\sup_{t\in J_k}\left(\frac{1}{|I_j|}\int_{I_j}|f_1(x+2^ktr)|^q~dr\right)^\frac{1}{q},
	\end{align*}
	where $I_j=[1-2^{-(j-1)},1-2^{-j}]$ and $2^kJ_k=E\cap[2^k,2^{k+1}]$. Note that $J_k\subseteq [1,2]$. For each $k$, let $\mathcal{E}_{0,k}^{j}$ be the collection of disjoint intervals
	of length $2^{-j-1}$ that cover $J_k$. Then, \Cref{def} tells us that $\sup_{k\in\Z}\#\mathcal{E}_{0,k}^{j}\lesssim 2^{j\beta}$. By a change of variable $tr\to r$, we have
	\begin{align*}
		\sup_{t\in J_k}\left(\frac{1}{|I_j|}\int_{I_j}|f_1(x+2^ktr)|^q~dr\right)^\frac{1}{q}&\leq\sup_{J\in \mathcal{E}_{0,k}^{j}}\sup_{t\in J}\left(\frac{1}{t|I_j|}\int_{tI_j}|f_1(x+2^kr)|^q~dr\right)^\frac{1}{q}\\
		&\lesssim\left(\sum_{J\in \mathcal{E}_{0,k}^{j}}\frac{1}{|\widetilde{I_{j,J}}|}\int_{\widetilde{I_{j,J}}}|f_1(x+2^kr)|^q~dr\right)^\frac{1}{q},
	\end{align*}
	where $\widetilde{I_{j,J}}=\{ab:a\in J, b\in I_j\}$ with $|\widetilde{I_{j,J}}|\approx |I_j|$, and we used $\ell_q\hookrightarrow\ell_{\infty}$ for the set $\mathcal{E}_{0,k}^{j}$ in the last inequality. For every \(k\in\mathbb Z\), enumerate the intervals in $\mathcal E_{0,k}^j$ as
	\[\mathcal E_{0,k}^j=\{J_{k,\nu}^j:1\leq\nu\leq 2^{j\beta}\}.\]
	Then, we have
	\begin{align*}
		\mathcal{N}f_1(x)\lesssim\sum_{j=1}^\infty 2^{-\frac{j}{2}}\left(\sum_{\nu=1}^{2^{j\beta}}\sup_{k\in \mathbb{Z}}\frac{1}{|\widetilde{I_{j,J}}|}\int_{\widetilde{I_{j,J}}}|f_1(x+2^kr)|^q~dr\right)^\frac{1}{q}.
	\end{align*}
	Thus, for $p>q$, the triangle inequality gives
	\begin{align*}
		\left\|\mathcal{N}f_1\right\|_{L^p}
		&\lesssim\sum_{j=1}^{\infty}2^{-\frac{j}{2}}\left(\sum_{\nu=1}^{2^{j\beta}}\left\|\sup_{k\in \mathbb{Z}}\left(\frac{1}{|\widetilde{I_{j,J}}|}\int_{\widetilde{I_{j,J}}}|f_1(x+2^kr)|^q~dr\right)^\frac{1}{q}\right\|_{L^p}^q\right)^{\frac{1}{q}}\\
		&\lesssim\sum_{j=1}^{\infty}2^{-j\left(\frac{1}{2}-\frac{\beta}{q}\right)}\log(2+c2^{j})\|f_1\|_{L^p},
	\end{align*}
	where we have used the $L^p$-estimates of shifted Hardy-Littlewood maximal function, see \cite[Theorem 4.1]{Muscalu} and \cite{book-Stein} for more details. The sum in the above inequality is finite if $q>2\beta$. This implies that $\mathcal{M}_E$ is bounded from $L^{p}(\R)\times L^{\infty}(\R)$ to $L^{p}(\R)$ for $p>\max\{2\beta,1\}$. Interchanging the roles of $f_1$ and $f_2$, we can obtain $L^{\infty}\times L^{p}\to L^{p}$-boundedness of $\mathcal{M}_E$ for $p>\max\{2\beta,1\}$.
	
	Interpolating the above estimates with trivial estimates from $\mathcal{M}_{full}$, we get the desired boundedness in \Cref{main2} when $\beta\geq\frac{1}{2}$. Moreover, if $\beta<\frac{1}{2}$, $\mathcal{M}_E$ maps $L^{p_1}(\R)\times L^{p_2}(\R)$ to $L^{p}(\R)$, for $1<p_1,p_2,p<\infty$.
	
	\subsection{Example.} For every $0\le\beta\le1$, there exists $E\subseteq [1,2]$ of Minkowski dimension $\beta$ such that the $\delta$-neighbourhood of $E$ satisfies
	\[|E_\delta|\gtrsim \delta^{1-\beta},\]
	for all sufficiently small $\delta>0$. Let $\{I_j\}_{j=1}^{N_\delta}$ be a collection of $\delta^{\frac{1}{2}}-$length intervals that cover $E$. Then, we have
	\[N_\delta\lesssim\delta^{-\frac{\beta}{2}}.\]
	Define the set $F=\cup_{j=1}^{N_\delta}4I_j$, where $4I_j$ is the interval of length $4\delta^{\frac{1}{2}}$ with the same center as $I_j$. Then $|F|\lesssim\delta^{\frac{1-\beta}{2}}$.
	
	Let $f_1(x)=\mathlarger{\chi}_{[-\delta,\delta]}$ and $f_2(x)=\mathlarger{\chi}_{F}(x)$. Then, for $x$ in $\delta-$neighbourhood of $E$, we have
	\[\mathcal{M}_E(f_1,f_2)(x)\gtrsim\delta^{\frac{1}{2}}.\]
	Thus, if $\mathcal{M}_E$ is bounded from $L^{p_1}\times L^{p_2}$ to $L^{p}$, then
	\begin{align*}
		\delta^{\frac{1}{2}}\delta^{\frac{1-\beta}{p}}\lesssim\|\mathcal{M}_E(f_1,f_2)\|_{L^p}\lesssim\|f_1\|_{L^{p_1}}\|f_2\|_{L^{p_2}}\lesssim\delta^{\frac{1}{p_1}}\delta^{\frac{1-\beta}{2p_2}}.
	\end{align*}
	Taking the limit $\delta\to 0$, we get that
	\[\frac{1}{p_1}+\frac{1-\beta}{2p_2}\leq\frac{1-\beta}{p}+\frac{1}{2}.\]
	Interchanging the roles of $f_1$ and $f_2$, we can get that
	\[\frac{1-\beta}{2p_1}+\frac{1}{p_2}\leq\frac{1-\beta}{p}+\frac{1}{2}.\]
	
	\begin{remark}
		If we assume $\beta\geq\frac{1}{2}$ in the above equations along with either $p_2=\infty$ or $p_1=\infty$, then we get that $p_1\geq2\beta$ or $p_2\geq2\beta$. Note that from our result, we recover sharp boundedness for this case except for the endpoint.
	\end{remark}

	\subsection*{Acknowledgements.} Surjeet Singh Choudhary is supported by the National Center for Theoretical Sciences through NSTC grant 114-2124-M-002-004. Chun-Yen Shen is supported in part by NSTC through grant 111-2115-M-002-010-MY5. Saurabh Shrivastava acknowledges the support from Anusandhan National Research Foundation (ANRF), India under the projects ANRF/ARG/2025/000940/MS. 
	
	\bibliography{biblio}
	
\end{document}